\documentclass[final,onefignum,onetabnum]{siamart190516}

\usepackage{caption}
\usepackage{float}
\usepackage{tikz}
\usepackage{amsmath, amsfonts, amssymb,mathrsfs}
\usepackage{amsmath,amssymb}
\usepackage{bm,color}
\usepackage{diagbox}
\newcommand{\bs}{\boldsymbol}

\ifpdf%
  \DeclareGraphicsExtensions{.eps,.pdf,.png,.jpg}
\else
  \DeclareGraphicsExtensions{.eps}
\fi
\usepackage{amsopn}
\usepackage{booktabs}
\usepackage{lipsum}
\usepackage{amsfonts}
\usepackage{graphicx}
\usepackage{epstopdf}
\usepackage{algorithmic}

\usepackage{diagbox}
\usepackage{array}
\newlength\colwidth
\newlength\secpad
\newlength\rowheight
\newcolumntype{f}{>{\centering\arraybackslash}m{\colwidth}@{\hskip 0.3em}}

\usepackage{multirow}

\ifpdf
  \DeclareGraphicsExtensions{.eps,.pdf,.png,.jpg}
\else
  \DeclareGraphicsExtensions{.eps}
\fi

\newcommand{\triplenorm}[1]{\ensuremath{|\!|\!| #1 |\!|\!|}}
 
\newsiamremark{remark}{Remark}
\newsiamremark{hypothesis}{Hypothesis}
\crefname{hypothesis}{Hypothesis}{Hypotheses}
\newsiamthm{claim}{Claim}

\headers{FE Framework of MFD for Maxwell}{J. H. Adler, C. Cavanaugh, X. Hu, L.T. Zikatanov}

\title{A Finite-Element Framework for a Mimetic Finite-Difference Discretization of Maxwell's Equations\thanks{Submitted to the editors November ??, 2020.
}}

\author{James H. Adler\thanks{Department of Mathematics, Tufts University, Medford, MA 02155
  (\email{james.adler@tufts.edu}, \email{casey.cavanaugh@tufts.edu}, \email{xiaozhe.hu@tufts.edu}).} \and Casey Cavanaugh\footnotemark[2] \and Xiaozhe Hu\footnotemark[2]
\and Ludmil T. Zikatanov \thanks{Department of Mathematics, Pennsylvania State University, University Park, PA 16802
  \hbox{(\email{ludmil@psu.edu}).}}}

\usepackage{amsopn}

\ifpdf
\hypersetup{
  pdftitle={A Finite-Element Framework for a Mimetic Finite-Difference Discretization of Maxwell's Equations},
  pdfauthor={J. H. Adler, C. Cavanaugh, X. Hu, L.T. Zikatanov}
}
\fi

\begin{document}

\maketitle

\begin{abstract}
   Maxwell's equations are a system of partial differential equations
   that govern the laws of electromagnetic induction.  We study a mimetic finite-difference (MFD) discretization of the equations which preserves important underlying physical properties. We show that,
    after mass-lumping and appropriate scaling, the MFD discretization is equivalent to a structure-preserving finite-element (FE) scheme.  
   This allows for a transparent analysis of the MFD method using the FE framework, and provides an avenue for the construction of efficient and robust linear solvers for the discretized system.  In particular, block preconditioners designed for FE formulations can be applied to the MFD system in a straightforward fashion.  We present numerical tests which verify the accuracy of the MFD scheme and confirm the robustness of the  preconditioners.
\end{abstract}



\section{Introduction}\label{sec:intro}

We consider the numerical solution of Maxwell's equations in a bounded connected domain $\Omega \in \mathbb{R}^3$:
\begin{align}
\frac{\partial \bm{B}}{\partial t} + \nabla \times \bm{E} &= \bm{0}, \quad \text{in} \ \Omega \times (0,T], \label{eq:FaradayFull} \\ 
\epsilon \frac{\partial \bm{E}}{\partial t} - \nabla \times \mu^{-1} \bm{B} &= -\bm{j}, \, \text{in} \ \Omega \times (0,T], \label{eq:AmpereFull} \\
\nabla \cdot \epsilon \bm{E} &= 0,\quad \text{in} \ \Omega\times (0,T], \label{eq:GaussEFull} \\
\nabla \cdot \bm{B} &= 0,\quad \text{in} \ \Omega\times (0,T], \label{eq:GaussBFull}
\end{align}
Here, $\bm{B}(\bm{x},t)$ and $\bm{E}(\bm{x},t)$ are the unknown magnetic and electric fields, $\epsilon(\bm{x})$ and $\mu(\bm{x})$ are the permittivity and permeability of the medium, respectively, and $\bm{j}(\bm{x},t)$ is the current density satisfying $\nabla \cdot \bm{j} = 0$. For simplicity, we choose $\epsilon = \mu = 1$, and impose homogeneous essential (Dirichlet) boundary conditions which model a perfect conductor:
\begin{equation}
\bm{B} \cdot \bm{n} \big|_{\partial \Omega} = 0, \quad \bm{n} \times \bm{E} \big|_{\partial \Omega} = \bm{0}. \label{eq:BCs}
\end{equation}
More general cases can be handled with straightforward modifications. In particular, the analysis presented in this paper still holds with non-constant $\epsilon(\bm{x})$ and $\mu(\bm{x})$, for example, by using a piecewise constant approximation for the parameters. 

The coupled equations (\ref{eq:FaradayFull})--(\ref{eq:AmpereFull}), Faraday's and Ampere's laws, model the interaction of the electric and magnetic fields, while the Gauss laws (\ref{eq:GaussEFull})--(\ref{eq:GaussBFull}), model the flux constraints of the individual fields. One of the major difficulties in numerically solving the Maxwell system is related to the constraints (\ref{eq:GaussEFull})  and (\ref{eq:GaussBFull}) as it is often necessary, by physical or other considerations, to have analogues of such identities on the discrete level.

To rectify this, structure-preserving discretizations are used to
guarantee that conservation laws on the continuous level still hold at
the discrete
level. While there are a  variety of such numerical methods, we focus in this paper on the relationship between mimetic finite-differences (MFD) and a structure-preserving mixed finite-element method (FEM) derived via finite-element exterior calculus (FEEC). 

The MFD method is defined by operators designed to ``mimic" the continuous level operators \cite{MFDBook,MFDart2,MFDart3,MFD_art}. This technique is straightforward to derive and, by drawing similarities to the continuous operators, it is quite easy to see that the discrete operators do in fact obey the continuous level properties. MFD is simple to implement, can be applied directly to the strong form of the PDE system, and has few mesh requirements (i.e., general polyhedral grids can be used). Like most finite-difference methods, however, the error estimates and convergence theory require high regularity; well-posedness is difficult to prove; and it is not always clear how to optimally solve the resulting linear system. 

On the other hand, structure-preserving finite-element methods for Maxwell's equations have been widely studied (eg. \cite{jxsolver,jxnumapprox,MonkFE, block2}). Using FEEC ~\cite{FEextcalc,FEextcalc2}, one can show that de Rham exact sequences are guaranteed on the discrete level. Furthermore, in the FEEC setting, the convergence theory and a priori error estimates are well-known, and showing well-posedness of the discretized weak form follows directly from Babu\v{s}ka--Brezzi theory~\cite{babuska-inf-sup, brezzi-inf-sup}. Moreover, the construction of efficient solvers is also well-developed~\cite{precond2, precond1, precond3, precond4, precond5, block1, block2}.

An important question that arises when using such discretizations for PDEs is how to efficiently solve the resulting linear system. One widely used method for solving Maxwell's system discretized by FEM is based on block preconditioners for Krylov methods, constructed using Schur complements.  More generally, block preconditioners for saddle point systems, a class that the full Maxwell system falls into, are widely studied \cite{precond2, precond1, precond3, precond4, precond5}, and robustness and efficiency results are well-established.  In general, to ensure that the iterative solver does not destroy the properties of the discretization, preconditioners must be developed that also preserve the operator properties at each time step. This is essential for ensuring that the resulting numerical solution obeys the PDE constraints throughout both the spatial and time domains.  Such preconditioners have been developed for the full Maxwell system, (\ref{eq:FaradayFull})--(\ref{eq:GaussBFull}), using mixed FEM discretizations \cite{jxsolver, block2}, as well as for the simplified time-harmonic form of Maxwell's equations \cite{timeharm2, timeharm1}. Another variety of electromagnetic applications comes in the form of magnetohydrodynamics, where a similar block-preconditioning approach can be used for a FEM divergence-free preserving discretization \cite{block1}. 

In this work, we adapt the preconditioners developed in \cite{jxsolver} for the Maxwell system with impedance boundary conditions.  As in \cite{jxsolver, jxnumapprox}, we consider a variation of (\ref{eq:FaradayFull})--(\ref{eq:GaussBFull}), where an auxiliary pressure variable, $p(\bm{x},t)$, is introduced:
\begin{align}
\frac{\partial \bm{B}}{\partial t} + \nabla \times \bm{E} & = \bm{0}, \quad \text{in} \ \Omega \times (0,T], \label{eq:Faraday} \\
\frac{\partial \bm{E}}{\partial t} - \nabla \times \bm{B} + \nabla p &= \bm{-j}, \, \text{in} \ \Omega \times (0,T], \label{eq:Ampere} \\
\frac{\partial p}{\partial t} + \nabla \cdot \bm{E} &= 0, \quad \text{in} \ \Omega \times (0,T]. \label{eq:Gauss}
\end{align}
It is straightforward to show that with suitable initial conditions, $p(\bm{x},0) = 0$ and \newline $\nabla \cdot \bm{B}(\bm{x},0) =0$, (\ref{eq:FaradayFull})--(\ref{eq:GaussBFull}) is equivalent to (\ref{eq:Faraday})--(\ref{eq:Gauss}). A structure preserving discretization is essential to allowing this form to be solved in place of the full Maxwell system. 

Ideally, we would like to have the ease and simplicity of the MFD method with all of the well-posedness and preconditioning theory that supports the FEM.  Similar to \cite{BokilMFD,FeMFD}, in this paper, we apply the MFD method to the Maxwell system, and then analyze it in a FE framework. By using mass-lumping schemes and scaled basis functions in the FEM, we show that the two methods yield equivalent linear systems. This equivalence is used to apply the FE theory to the MFD system to show well-posedness of the mimetic discretization.  Additionally, since robust block preconditioners have been designed for the Maxwell FE system that preserve the constraints at all time iterations \cite{jxsolver}, we demonstrate how to slightly modify those results to obtain robust linear solvers for the MFD system of Maxwell's equations.

This paper is organized as follows. In Section 2, we introduce the notation and discretization technique for MFD and apply it to the Maxwell system. Section 3 recalls the FE discretization and presents a mass-lumped alternative. Then, Section 4 draws connections between the two methods for Maxwell, shows their equivalence, and well-posedness of the MFD system is proven. Section 5 introduces and analyzes block preconditioners for the MFD system, proving their robustness.  Finally, Section 6 presents numerical results to demonstrate the theoretical results, and concluding remarks and future work are discussed in Section 7.

\section{The Mimetic Finite-Difference Method}\label{sec:discretizations}

Following \cite{MFD_art}, we construct a primal (Delaunay) tetrahedral mesh and a dual (Voronoi) polyhedral grid. Denote the vertices/nodes of the Delaunay triangulation by $\{ \bm{x}_i^D \}_{i=1}^{N_D}$, and an edge on the Delaunay mesh connecting nodes $\bm{x }_i^D$ and $\bm{x}_j^D$ by $\bm{e}_{ij}^D$, with unit tangent vector $\bm{t}_{ij}^D$ pointing from vertex of lower index to vertex of higher index. The Delaunay tetrahedra are given by $D_k, \, k = 1,...,N_V$. Each  $D_k$ has face set (boundary) $\partial D_k$. The neighbor set of tetrahedron $D_k$, given by $\mathcal{N}_k^D$, is defined as the set of indices of the tetrahedra that share common planes with $D_k$, i.e., $\mathcal{N}_k^D := \{ m \, \big| \, \partial D_k \cap \partial D_m \neq \emptyset, \, m = 1,...,N_V \}$.
The common plane (face) between $D_k$ and $D_m$ is given by $\partial D_{km}$. For tetrahedron $D_k$ with face $\partial D_{km}$, we define the unit outward normal vector $\bm{n}_{km}^D$ pointing outward from $D_k$.
Analogous to the above definitions, we have for the dual Voronoi mesh, nodes, $\{\bm{x}_k^V\}_{k=1}^{N_V}$, edges, $\bm{e}_{km}^V$, with unit tangent vector $\bm{t}_{km}^V$, polyhedra, $V_i$, $i = 1,...,N_D$, face set, $\partial V_i$, neighbor set, $\mathcal{N}_i^V$, common plane $\partial V_{ij}$, and outward unit normal vector $\bm{n}_{ij}^V$.

This dual mesh configuration yields some useful properties that are exploited when defining the MFD operators. The Voronoi point $\bm{x}_k^V$ is the circumcenter of the Delaunay tetrahedron $D_k$.  Additionally, the Delaunay nodes define the Voronoi polyhedra. We define $V_i$ as the set of points in the domain that lie closer to Delaunay node $\bm{x}_i^D$ than any other Delaunay node, 
\[V_i := \{ \bm{x} \in \Omega \, \big| \, |\bm{x} - \bm{x}_i^D | \leq | \bm{x} - \bm{x}_j^D|, \, j = 1,...,N_D, \, j\neq i \}.\]
Furthermore, we have that each Delaunay edge $\bm{e}_{ij}^D$ is orthogonal to the Voronoi face $\partial V_{ij}$, and each Voronoi edge $\bm{e}_{km}^V$ is orthogonal to Delaunay face, $\partial D_{km}$. This gives us a one-to-one correspondence between nodes on one mesh to polyhedra on the other, and edges on one mesh to faces on the other. Figure \ref{fig:dualmeshes} illustrates an example mesh in two dimensions to further highlight the notation used.
Note that this dual mesh configuration with these properties requires that the circumcenters of the Delaunay triangulation lie in the interior of the Delaunay tetrahedra. While this is not a strict requirement for the MFD method to work, it allows for simplicity in the analysis and implementation (see \cite{MLBrezzi,MFD_art}). 
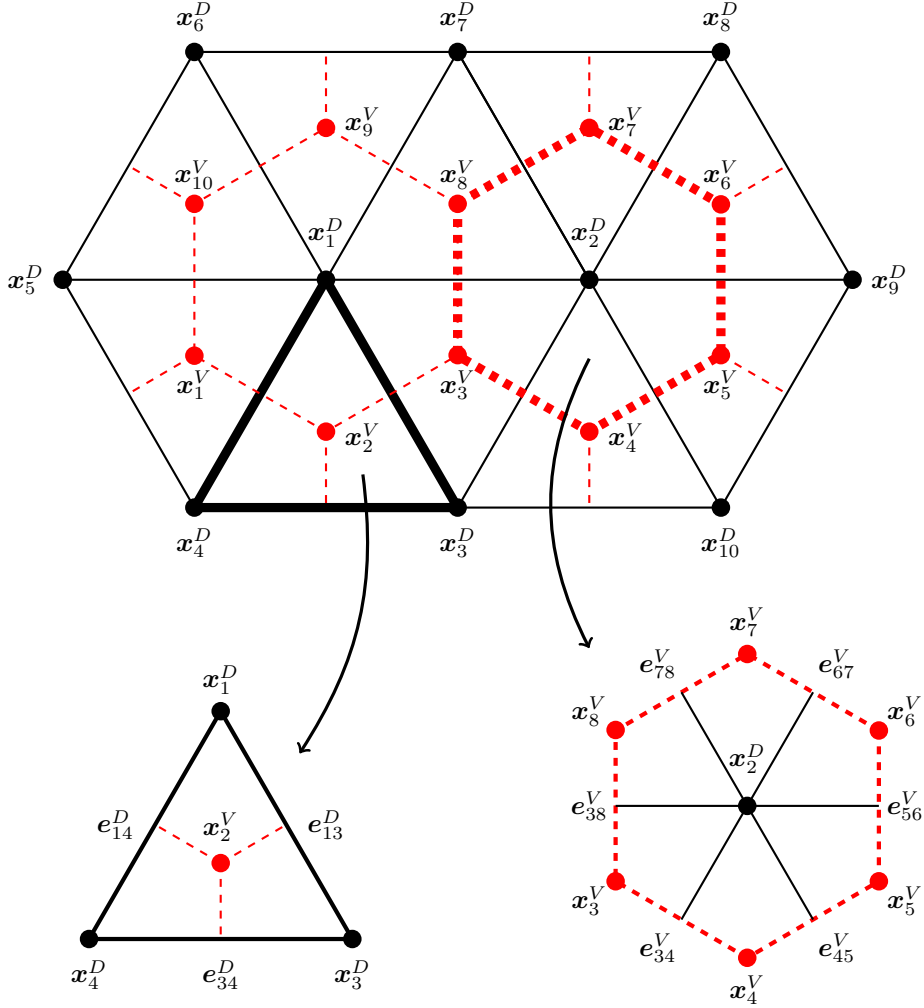
\begin{figure}[!htb] 
  \centering
     \begin{tikzpicture}[scale = 7, every node/.style={scale = 1}]
\draw[thick] (.5,0) to (.25, .433) to (-.25, .433) to (-.5, 0) to (-.25, -.433) to (.25, -.433) to cycle;
\draw[thick] (.25, .433) to (-.25, -.433);
\draw[thick] (-.25, .433) to (.25, -.433);
\draw[thick] (.5,0) to (-.5, 0);
\draw[thick] (.25, .433) to (.75, .433) to (1,0) to (.75, -.433) to (.5, 0) to cycle;
\draw[thick] (.75, -.433) to (.25, -.433);
\draw[thick] (.5,0) to (1,0);
\draw[thick] (.5,0) to (.75, .433);

\draw[line width=3.5] (0,0) to (.25, -.433) to (-.25, -.433) to cycle;

\draw[dashed, red, thick] (.25, .1443) to (.25, -.1433) to (0, -.2887) to (-.25, -.1443) to (-.25, .1443) to (0, .2887) to cycle;
\draw[dashed, red , thick] (0, .2887) to (0, .433);
\draw[dashed, red, thick] (0, -.2887) to (0, -.433);
\draw[dashed, red, thick] (-.25, .1433) to (-0.375, 0.2165);
\draw[dashed, red, thick] (.25, .1433) to (0.5, .2887);
\draw[dashed, red, thick] (-.25, -.1433) to (-0.375, -0.2165);
\draw[dashed, red, thick](.5, .2887) to (.5, .433);
\draw[dashed, red, thick](.5, -.2887) to (.5, -.433);
\draw[dashed, red, thick] (.75, .1433) to (.875  , 0.2165);
\draw[dashed, red, thick] (.75, -.1433) to (.875  , -0.2165);

\draw[dashed, red, line width=3.5] (.25, .1443) to (0.5, .2887) to  (.75, .1433) to (.75, -.1433) to (0.5, -.2887) to (.25, -.1443) to cycle;

\fill (0,0) circle (.5pt);
\fill (.5,0) circle (.5pt);
\fill (.25, .433) circle (.5pt);
\fill (-.25, .433) circle (.5pt);
\fill (-.5, 0) circle (.5pt);
\fill (-.25, -.433) circle (.5pt);
\fill (.25, -.433) circle (.5pt);

\fill (.75, .433) circle (.5pt);
\fill (.75, -.433) circle (.5pt);
\fill (1,0) circle (.5pt);

\fill[red] (.25, .1443) circle (.5pt);
\fill[red] (.25, -.1433) circle (.5pt);
\fill[red] (0, -.2887) circle (.5pt);
\fill[red] (-.25, -.1443) circle (.5pt);
\fill[red] (-.25, .1443) circle (.5pt);
\fill[red] (0, .2887) circle (.5pt);

\fill[red] (.75, .1433) circle (.5pt);
\fill[red] (.75, -.1433) circle (.5pt);
\fill[red] (.5, -0.2887) circle (.5pt);
\fill[red] (0.5, .2887) circle (.5pt);

\node at (0, .09) {$\bm{x}_1^D$};
\node at (.5, .09) {$\bm{x}_2^D$};
\node at (-.57, 0) {$\bm{x}_5^D$};
\node at (.25, .5){$\bm{x}_7^D$};
\node at (-.25, .5){$\bm{x}_6^D$};
\node at (-.25, - .5){$\bm{x}_4^D$};
\node at (.25, -.5){$\bm{x}_3^D$};
\node at(.75, .5){$\bm{x}_8^D$};
\node at (.75, -.5) {$\bm{x}_{10}^D$};
\node at (1.07, 0) {$\bm{x}_9^D$};

\node at (.25, -.2) {$\bm{x}_3^V$};
\node at ( .07, -.3){$\bm{x}_2^V$};
\node at (-.25, -.2) {$\bm{x}_1^V$};
\node at (-.25, .2) {$\bm{x}_{10}^V$};
\node at ( .07, .3){$\bm{x}_9^V$};
\node at (.25, .2) {$\bm{x}_8^V$};
\node at (.75, -.2){$\bm{x}_{5}^V$};
\node at (.75, .2) {$\bm{x}_6^V$};
\node at ( .57, -.3){$\bm{x}_{4}^V$};
\node at ( .57, .3){$\bm{x}_{7}^V$};

\begin{scope}[yshift = -.82cm, xshift = -.2cm]
[scale = 7, every node/.style={scale = 1}]
\draw[ultra thick] (0,0) to (-.25, -.433) to (.25, -.433) to cycle;
\draw[dashed, red, thick] (0, -.2887) to (0, -.433);
\draw[dashed, red, thick] (0, -.2887) to (.125, -.2155);
\draw[dashed, red, thick] (0, -.2887) to (-.125, -.2155);

\fill (0,0) circle (.5pt);
\fill (-.25, -.433) circle (.5pt);
\fill (.25, -.433) circle (.5pt);
\fill[red] (0, -.2887) circle (.5pt);

\node at (0, .06) {$\bm{x}_1^D$};
\node at (-.25, - .5){$\bm{x}_4^D$};
\node at (.25, -.5){$\bm{x}_3^D$};
\node at (-.2, -.2155){$\bm{e}_{14}^D$};
\node at (.2, -.2155){$\bm{e}_{13}^D$};
\node at (0, -.5) {$\bm{e}_{34}^D$};
\node at ( .0, -.22){$\bm{x}_2^V$};

\end{scope}

\draw[->, very thick](.07, -.37) to[bend left = 20] (-.05, -.9);

\begin{scope}[yshift = -1cm, xshift = .3cm]
\draw[dashed, red, ultra thick] (.75, .1433) to (.75, -.1433) to (.5, -.2887) to (.25, -.1433) to (.25, .1433) to (.5, .2887) to cycle;
\draw[thick] (.5,0) to (.625, .216); 
\draw[thick] (.5,0) to (.625, -.216);
\draw[thick] (.5,0) to (.375, -.216);
\draw[thick] (.5,0) to (.375, .216);
\draw[thick] (.5,0) to (.75,0);
\draw[thick] (.5, 0) to (.25,0);

\fill (.5,0) circle (.5pt);
\fill[red] (.25, .1443) circle (.5pt);
\fill[red] (.25, -.1433) circle (.5pt);
\fill[red] (.75, .1433) circle (.5pt);
\fill[red] (.75, -.1433) circle (.5pt);
\fill[red] (.5, -0.2887) circle (.5pt);
\fill[red] (0.5, .2887) circle (.5pt);

\node at (.5, .09) {$\bm{x}_2^D$};
\node at (.2, -.18) {$\bm{x}_3^V$};
\node at (.2, .18) {$\bm{x}_8^V$};
\node at (.8, -.18){$\bm{x}_{5}^V$};
\node at (.8, .18) {$\bm{x}_6^V$};
\node at ( .5, -.35){$\bm{x}_{4}^V$};
\node at ( .5, .35){$\bm{x}_{7}^V$};
\node at (.33, .27) {$\bm{e}_{78}^V$};
\node at (.33, -.27) {$\bm{e}_{34}^V$};
\node at (.67, .27) {$\bm{e}_{67}^V$};
\node at (.67, -.27) {$\bm{e}_{45}^V$};
\node at (.2, 0) {$\bm{e}_{38}^V$};
\node at (.8, 0) {$\bm{e}_{56}^V$};

\end{scope}
\draw[->, very thick](.50, -.15) to[bend right = 27] (.5, -.7);

\end{tikzpicture}  

     \captionsetup[figure]{labelfont={it}}
  \caption{ Top: Two-dimensional primal Delaunay mesh in black solid lines with corresponding dual Voronoi mesh in red dashed lines. Bottom Left: Zoom in of Delaunay element $D_2$, corresponding to Voronoi node $\bm{x}_2^V$, with labeled Delaunay edges and nodes. Bottom Right: Zoom in of Voronoi element $V_2$, corresponding to Delaunay node $\bm{x}_2^D$, with labeled Voronoi nodes and edges.} \label{fig:dualmeshes}
\end{figure}

\subsection{Grid Functions and MFD Operators}
Following~\cite{MFDart2,MFD_art}, we define functions and operators on both the Delaunay and Voronoi meshes. First, approximations of scalar functions in the domain are represented with scalar grid functions that are defined on the nodes of the meshes. Thus, scalar functions defined on the Delaunay nodes are constants on Voronoi polyhedra, and scalar functions on Voronoi nodes are constants on Delaunay tetrahedra. The corresponding function spaces are as follows,
\begin{align}
H_D &:= \{ u(\bm{x}) \, \big| \, u(\bm{x}) = u(\bm{x}_i^D) = u_i^D, \,\, \forall \bm{x} \in V_i, \, i = 1,...,N_D \}, \label{eq:scalarD}  \\
H_V &:= \{ u(\bm{x}) \, \big| \, u(\bm{x}) = u(\bm{x}_k^V) = u_k^V, \,\, \forall \bm{x} \in D_k, \, k = 1,...,N_V \}. \label{eq:scalarV}
\end{align}
Vector functions are approximated on the Delaunay mesh with vector grid functions, where the function space is denoted by $\bm{H_D}$. For vector function $\bm{u}(\bm{x})$, we project it along the Delaunay edge, and evaluate at the intersection of the Delaunay edges and Voronoi face.  The space of vector grid functions on the Voronoi mesh, $\bm{H_V}$, is defined analogously on the Voronoi mesh,
\begin{align}
\bm{H_D} &:= \{ \bm{u} (\bm{x}) \, \big| \, \bm{u} (\bm{x}) = \bm{u} \cdot \bm{t}_{ij}^D (\bm{x}_{ij}^D) = u_{ij}^D, \,\, \bm{x}_{ij}^D = \bm{e}_{ij}^D \cap \partial V_{ij} \}, \\
\bm{H_V} &:= \{ \bm{u} (\bm{x}) \, \big| \, \bm{u} (\bm{x}) = \bm{u} \cdot \bm{t}_{km}^V (\bm{x}_{km}^V) = u_{km}^V, \,\, \bm{x}_{km}^V = \bm{e}_{km}^V \cap \partial D_{km} \}.
\end{align}

To build intuition, we first introduce the MFD operators component-wise, then later give the matrix definitions. 
In the continuous setting, the gradient maps scalar functions to vector functions. Analogously, the discrete gradient on the Delaunay mesh maps a scalar grid function defined on the nodes to a vector grid function defined on the edges, or $\text{grad}_D u : H_D \to \bs{H_D}$. On edge $\bm{e}_{ij}^D$,  
\begin{equation*}
(\text{grad}_D u )_{ij}^D = \frac{u_j^D - u_i^D}{|\bm{e}_{ij}^D|} \,\, \eta(i,j),
\end{equation*}
where $\eta$ is an orientation constant,
\begin{equation*}
\eta (i,j) = \begin{cases} \,\,\,\,1, & j>i \\
						 -1, & \text{otherwise.}
			\end{cases}
\end{equation*}			  
Similarly, the gradient on the Voronoi mesh, $\text{grad}_V u : H_V \to \bs{H_V}$, is given by,
\begin{equation*}
(\text{grad}_V u )_{km}^V = \frac{u_m^V - u_k^V}{|\bm{e}_{km}^V|} \,\, \eta(k,m).
\end{equation*}
To define the discrete divergence, first note that divergence maps vector functions to scalar functions. This differential operator on the Delaunay grid, $\text{div}_D u : \bs{H_D} \to H_D$, corresponding to the outward flux of $V_i$ is defined as
\begin{equation*}
(\text{div}_D u)_i^D = \frac{1}{|V_i|} \sum_{j \in \mathcal{N}_i^V} |\partial V_{ij}| u_{ij}^D (\bs n_{ij}^V \cdot \bs t_{ij}^D).
\end{equation*}
\noindent
Similarly, on the Voronoi grid, the divergence $\text{div}_V u: \bs{H_V} \to H_V$ is
\begin{equation*}
(\text{div}_V u)_k^V = \frac{1}{|D_k|} \sum_{m \in \mathcal{N}_k^D} |\partial D_{km}| u_{km}^V (\bs n_{km}^D \cdot \bs t_{km}^V).
\end{equation*}
The discrete curl operator maps from edges on one mesh (the circulation) to edges on the other mesh (the axis of rotation) by the geometric relationships between the dual meshes. Therefore, the Delaunay $\operatorname{curl}$ operator maps vector grid functions on the Delaunay mesh to a vector grid function on the Voronoi mesh, $\text{curl}_D  u: \bs{H_D} \to \bs{H_V}$ and is given by, 

\begin{equation*}
(\text{curl}_D  u)_{km}^V = \frac{(\bs t_{km}^V \cdot \bs n_{km}^D)}{|\partial D_{km}|} \sum_{\bm{e}_{ij}^D \in \partial D_{km}} u_{ij}^D \,\, |\bm{e}_{ij}^D| \,\, \chi(\bs n_{km}^D, \bs t_{ij}^D),
\end{equation*}
where the constant $\chi$ essentially enforces the right-hand rule,
\begin{equation*}
\chi(\bs n_{km}^D, \bs t_{ij}^D) = \begin{cases}
\,\,\,\,1, & \bs t_{ij}^D \text{ positively oriented,} \\
-1, & \text{otherwise.}
\end{cases}
\end{equation*}
Similarly, we define the Voronoi curl operator $\text{curl}_V  u: \bs{H_V} \to \bs{H_D}$,
\begin{equation*}
(\text{curl}_V  u)_{ij}^D = \frac{(\bs t_{ij}^D \cdot \bs n_{ij}^V)}{|\partial V_{ij}|} \sum_{\bm{e}_{km}^V \in \partial V_{ij}} u_{km}^V \,\, |\bm{e}_{km}^V| \,\, \chi(\bs n_{ij}^V, \bs t_{km}^V).
\end{equation*}

To define the MFD operators in matrix form, we introduce the edge-vertex signed incidence matrix, $\mathcal{G} \in \mathbb{R}^{M_D \times N_D},$ and the face-edge signed incidence matrix, $\mathcal{K} \in \mathbb{R}^{M_V \times M_D}$. Both are defined on the Delaunay triangulation where $N_D$, $M_V$ and $M_D$ denote the number of Delaunay nodes, Voronoi edges, and Delaunay edges, respectively. Similarly, on the Voronoi mesh, we have the signed incidence matrix $\mathcal{G}_V \in \mathbb{R}^{M_V \times N_V}$, where $N_V$ denotes the number of nodes on the Voronoi mesh. The nonzero entries of $\mathcal{G}$, $\mathcal{G}_V$, and $\mathcal{K}$ are either $1$ or $-1$, and the signs are consistent with the pre-determined orientation of the edges and faces. Additionally, we introduce the following diagonal matrices encoding the mesh information pertaining to MFD, \\
\begin{center}
\begin{tabular}{ccc}
$\mathcal{D}_{\bm{e}^D} = \text{diag} \left(|\bm{e}_{ij}^D| \right)$, &  
$\mathcal{D}_{\partial D} = \text{diag} \left(|\partial D_{km}| \right)$, & $\mathcal{D}_{D} = \text{diag} \left(|D_k| \right)$, 
\\ \\
$\mathcal{D}_{\bm{e}^V} = \text{diag} \left(|\bm{e}_{km}^V| \right)$,
& $\mathcal{D}_{\partial V} = \text{diag} \left(|\partial V_{ij}| \right)$, 
&
$\mathcal{D}_{V} = \text{diag} \left(|V_i| \right)$. \\ \\
\end{tabular}
\end{center}
The matrix representations are derived from the component-wise definitions using the incidence matrices as the actions of the operators, and the diagonal matrices for the appropriate scaling.  Thus,
\begin{align}
 \label{eq:graddivcurlD}
 \text{grad}_D &:= \mathcal{D}_{\bm{e}^D}^{-1} \mathcal{G},& 
 \text{div}_D &:= \mathcal{D}_V^{-1} \mathcal{G}^T \mathcal{D}_{\partial V},& \text{curl}_D &:= \mathcal{D}_{\partial D}^{-1} \mathcal{K} \mathcal{D}_{\bm{e}^D},\\
\label{eq:graddivcurlV}
\text{grad}_V &:= \mathcal{D}_{\bm{e}^V}^{-1} \mathcal{G}_V, &
\text{div}_V &:= \mathcal{D}_D^{-1} \mathcal{G}_V^T \mathcal{D}_{\partial D}, &
\text{curl}_V &:= \mathcal{D}_{\partial V}^{-1} \mathcal{K}^T \mathcal{D}_{\bm{e}^V}.
\end{align}

With this construction, it is known that the mimetic operators are structure-preserving, i.e., $\text{curl}_D \text{grad}_D = 0$, $\text{curl}_V \text{grad}_V = 0$, $\text{div}_V \text{curl}_D = 0$, and $\text{div}_D \text{curl}_V = 0$ \cite{MFDBook,MFD_art}. Using these relationships, two exact sequences exist for MFD,
\begin{align}
H_D &\xrightarrow{\text{grad}_D} \bm{H_D} \xrightarrow{\text{curl}_D} \bm{H_V} \xrightarrow{\text{div}_V} H_V, \label{eq:MFDderham1} \\
H_V &\xrightarrow{\text{grad}_V} \bm{H_V} \xrightarrow{\text{curl}_V} \bm{H_D} \xrightarrow{\text{div}_D} H_D.\label{eq:MFDderham2}
\end{align}
Note again that the nature of the discretization method, when used for Maxwell's system, enforces the PDE constraints strongly at the discrete level.

\subsection{MFD for Maxwell's Equations}
Since the energy conservation property is important in electromagnetic applications, we consider the Crank--Nicolson scheme (a symplectic time integrator) with time-step $\tau$ and suitable initial conditions given by appropriate interpolation to the dual meshes, $p_D^0$, $\bm{E}_D^0$, and $\bm{B}_V^0$. The fully discretized system becomes: find $p^n_D \in H_D$, $\bm{E}^n_D \in \bm{H_D}$, and $\bm{B}^n_V \in \bm{H_V}$ such that,

\begin{align}
\frac{2}{\tau} \bm{B}_V^n + \text{curl}_D \, \bm{E}_D^n & =  \bm{g}^V_{\bm{B}},  \label{eq:MFDFaraday}  \\
\frac{2}{\tau} \bm{E}_D^n - \text{curl}_V \, \bm{B}_V^n + \text{grad}_D \, p_D^n & =  \bm{g}^D_{\bm{E}}, \label{eq:MFDAmpere} \\
\frac{2}{\tau} p_D^n + \text{div}_D \, \bm{E}_D^n & =  g^D_p, \label{MFDGauss}
\end{align}
where the functions on the right-hand sides are given by,
\begin{align*}
\bm{g}^V_{\bm{B}} &:= \frac{2}{\tau} \bm{B}_V^{n-1} - \text{curl}_D \, \bm{E}_D^{n-1}, \\
\bm{g}^D_{\bm{E}} & := \frac{2}{\tau} \bm{E}_D^{n-1} + \text{curl}_V \, \bm{B}_V^{n-1} - \text{grad}_D \,p_D^{n-1} - (\bm{j}_D^n + \bm{j}_D^{n-1}), \\
g^D_p & := \frac{2}{\tau} p_D^{n-1} - \text{div}_D \, \bm{E}_D^{n-1}.
\end{align*}
The current density, $\bm{j}_D \in \bm{H_D}$, is $\left(\bm{j}_D \right)_{ij}^D = |\partial V_{ij}^D|^{-1} \int_{\partial V_{ij}} \bm{j} \cdot \bm{n}_{ij}^V \, d\bm{x}$. 
\begin{remark}Recall that  $p(\bm{x},t)=0$ for all $\bm{x}\in \Omega$ and $t \geq 0$ with initial condition $p(\bm{x},0) = 0$. Therefore, (\ref{eq:MFDFaraday})--(\ref{MFDGauss}) could be solved without including $p$ and the analysis that follows remains the same even without $p$.  However, $p$ is included to demonstrate the relationship between the $\operatorname{grad}$, $\operatorname{curl}$, and $\operatorname{div}$ spaces, which are chosen such that the equations obey the mappings given by the sequences (\ref{eq:MFDderham1})--(\ref{eq:MFDderham2}). Also, note that we could have put $\bm{E}$ and $p$ on the Voronoi mesh, and $\bm{B}$ on the Delaunay mesh instead of the choice above. However, by putting the magnetic field on the Voronoi mesh, $\bm{B}_V \in \bm{H_V}$, guarantees $\text{div}_V \, \bm{B}_V = 0$. Thus, the divergence of the magnetic field is constant zero on the Voronoi nodes which gives us a divergence-free magnetic field on all Delaunay tetrahedra. 
\end{remark}
Using the definitions of the MFD operators, the linear system for the MFD scheme (\ref{eq:MFDFaraday})--(\ref{MFDGauss}) is given by
\begin{equation}
 \underbrace{\begin{bmatrix}
 \frac{2}{\tau} \mathcal{I}_{\bm{e}^V} & \mathcal{D}_{\partial D}^{-1} \, \mathcal{K} \, \mathcal{D}_{\bm{e}^D}  &    \\
 - \mathcal{D}_{\partial V}^{-1} \, \mathcal{K}^T \, \mathcal{D}_{\bm{e}^V}   &  \frac{2}{\tau} \mathcal{I}_{\bm{e}^D}  & \mathcal{D}_{\bm{e}^D}^{-1} \, \mathcal{G} \\
  	& -\mathcal{D}_{V}^{-1} \, \mathcal{G}^T \, \mathcal{D}_{\partial V} & \frac{2}{\tau} \mathcal{I}_V 
 \end{bmatrix}}_{=:\mathcal{A}_{\text{MFD}}}
 \begin{bmatrix}
 \bm{B}_V^n \\
 \bm{E}_D^n \\
 p_D^n
 \end{bmatrix}
 =
  \begin{bmatrix}
 \bm{g}^V_{\bm{B}} \\
\bm{g}^D_{\bm{E}} \\
 g^D_p
 \end{bmatrix}. \label{eq:MFDLinear}
\end{equation}

\section{Finite-Element Framework}\label{sec:FEM} 
Next, we consider a structure-preserving\newline mixed FEM for the Maxwell system \cite{FEextcalc}. To approximate the inner-product terms on the computational domain, we implement mass lumping, which results in diagonal mass matrices. This gives us FE blocks in terms of MFD mesh information, which is useful when drawing connections in the next section.

Consider the differential operator, $\mathfrak{D}$, and Sobolev space, 
\begin{equation*}
H(\mathfrak{D}) := \{ u \in L^2 (\Omega), \mathfrak{D}u \in L^2(\Omega) \},
\end{equation*}
where $\mathfrak{D}$ is grad, curl, or div. Let $\| \cdot \|$ and $\langle \cdot, \cdot \rangle$ denote the $L^2$ norm and inner product, respectively. Define the finite-dimensional function spaces with appropriate boundary conditions, $H_{h,0}(\mathfrak{D})$. For the magnetic field, $\bm{B}_h\in\bm{H}_{h,0} (\operatorname{div})$, we use the Raviart--Thomas (RT) element, the N\'ed\'elec element for the electric field $\bm{E}_h\in\bm{H}_{h,0}(\text{curl})$, and the Lagrange element for the auxiliary pressure $p_h\in H_{h,0} (\operatorname{grad})$. Defining $\bm{V}_h := \bm{H}_{h,0} (\operatorname{div}) \times \bm{H}_{h,0}(\text{curl}) \times H_{h,0} (\operatorname{grad})$,
the FEM discretization for Maxwell's equations becomes, find ($\bm{B}_h^n, \bm{E}_h^n, p_h^n) \in \bm{V}_h$ such that for all $(\bm{C}_h, \bm{F}_h, q_h ) \in \bm{V}_h$,
\begin{align}
\frac{2}{\tau}\langle \bm B_h^n, \bm C_h \rangle + \langle \nabla \times \bm E_h^n, \bm C_h \rangle &= (\bm{g_B}, \bm C_h), \label{eq:FEMFaraday} \\
\frac{2}{\tau} \langle \bm E_h^n, \bm F_h \rangle - \langle \bm B_h^n, \nabla \times \bm F_h \rangle  + \langle \nabla p_h^n, \bm F_h \rangle &= (\bm{g_E}, \bm F_h),  \label{eq:FEMAmpere} \\
\frac{2}{\tau} \langle p_h^n, q_h \rangle - \langle \bm E_h^n, \nabla q_h \rangle &= (g_p, q_h), \label{eq:FEMGauss}
\end{align} 
where the functionals on the right-hand side are given by,
\begin{align*}
(\bm{g}_{\bm{B}},\bm{C}_h) &= \frac{2}{\tau}\langle \bm{B}_h^{n-1},\bm{C}_h \rangle - \langle \nabla \times \bm{E}_h^{n-1}, \bm{C}_h \rangle, \\
(\bm{g}_{\bm{E}},\bm{F}_h) &= \frac{2}{\tau} \langle \bm{E}_h^{n-1}, \bm{F}_h \rangle - \langle \nabla p_h^{n-1}, \bm{F}_h \rangle + \langle \bm{B}_h^{n-1}, \nabla \times \bm{F}_h \rangle - \langle \bm{j}^n+\bm{j}^{n-1}, \bm{F}_h \rangle,\\
(g_p,q_h) &= \frac{2}{\tau} \langle p_h^{n-1}, q_h \rangle + \langle \bm{E}_h^{n-1}, \nabla q_h \rangle.
\end{align*}

To write this as a linear system, we introduce discrete gradient and curl operators. Let $\{ \phi_i^{\operatorname{grad}} \}$, $\{\bm{\phi}_i^{\operatorname{curl}} \}$, and $\{\bm{\phi}_i^{\operatorname{div}} \}$ be the bases of $H_{h,0}(\operatorname{grad})$, $\bm{H}_{h,0} (\operatorname{curl})$, and  $\bm{H}_{h,0} (\operatorname{div})$, respectively. Let $\{\bm{\eta}_i^{\operatorname{curl}}\}$ be the degrees of freedom for the N\'ed\'elec space,  and $\{ \bm{\eta}_i^{\operatorname{div}}\}$  for the RT space. Then, define the FE discrete gradient, $\mathcal{G}^{FE}$, and curl, $\mathcal{K}^{FE}$, in terms of the degrees of freedom,
\begin{align}
\mathcal{G}^{FE}_{ij} &:= \bm{\eta}_i^{\operatorname{curl}} \left( \nabla \phi_j^{\operatorname{grad}} \right) = \frac{1}{|e_{i}|} \int_{e_{i}} \nabla \phi^{\operatorname{grad}}_j \cdot \bm{t}_{i} \, ds, \label{eq:gradFEdof} \\
\mathcal{K}^{FE}_{ij} &:= \bm{\eta}_i^{\operatorname{div}} \left( \nabla \times \bm{\phi}_j^{\operatorname{curl}} \right) = \frac{1}{|f_i|} \int_{f_i} \nabla \times \bm{\phi}_j^{\operatorname{curl}} \cdot \bm{n}_i \, dS. \label{eq:curlFEdof}
\end{align}
Note that the degrees of freedom are scaled by mesh data; the $\operatorname{curl}$ degrees of freedom are scaled by inverse edge lengths, and $\operatorname{div}$ degrees of freedom are scaled by inverse face areas.
Thus, the FE linear system for (\ref{eq:FEMFaraday})--(\ref{eq:FEMGauss}) is given by

\begin{equation}
\underbrace{\begin{bmatrix}
\frac{2}{\tau} \mathcal{M}_{\bm{B}} & \mathcal{M}_{\bm{B}} \mathcal{K}^{FE} &   \\
- \left( \mathcal{K}^{FE} \right)^T \mathcal{M}_{\bm{B}} & \frac{2}{\tau} \mathcal{M}_{\bm{E}} & \mathcal{M}_{\bm{E}} \mathcal{G}^{FE} \\
   	& - \left( \mathcal{G}^{FE} \right)^T \mathcal{M}_{\bm{E}} & \frac{2}{\tau} \mathcal{M}_p  
\end{bmatrix}}_{=:\mathcal{A}_{\text{FE}}}  
\begin{bmatrix}
\bm{B}_h^n \\
\bm{E}_h^n \\
p_h^n
\end{bmatrix}
= 
\begin{bmatrix}
\bm{g}_{\bm{B}}\\
\bm{g}_{\bm{E}} \\
g_{p}
\end{bmatrix}, \label{eq:FELinear1}
\end{equation}
with mass matrices given by $\left(\mathcal{M}_p \right)_{ij} = \langle \phi_j^{\operatorname{grad}}, \phi_i^{\operatorname{grad}} \rangle$, $\left(\mathcal{M}_{\bm{E}} \right)_{ij} = \langle \bm{\phi}_j^{\operatorname{curl}}, \bm{\phi}_i^{\operatorname{curl}} \rangle$, and  $\left(\mathcal{M}_{\bm{B}} \right)_{ij} = \langle \bm{\phi}_j^{\operatorname{div}}, \bm{\phi}_i^{\operatorname{div}} \rangle$.
\subsection{Mass Lumping}\label{subsec:masslumping}
Now that we have linear systems for two discretization methods, our goal is to draw similarities between (\ref{eq:MFDLinear}) and (\ref{eq:FELinear1}). Notice that the discrete differential operators are already in the same blocks; however, the MFD system has blocks with diagonal entries containing mesh information while the FEM system has mass matrices. Using ideas from \cite{Nedlump, RTlump, MLBrezzi, MLBook} to approximate the mass matrices with diagonal matrices, we implement mass-lumping schemes. The quadrature schemes associated with the lumping result in a quantity that has physical significance, namely a volume unit for 3D meshes, and an area on the mesh for the 2D case. Therefore, we expect that mass lumping gives diagonal matrices where the diagonal represents some sort of volume (3D) or area (2D) in the dual mesh set-up. We see in the next section that this is true for appropriate choices of quadrature weights, and that these specific lumped matrices help draw more connections between the FE and MFD systems.

First, consider the Lagrange element mass matrix,  $\mathcal{M}_p$. For simplicity, assume we have a regular primal (Delaunay) mesh for the finite-element grid, and an associated dual (Voronoi) mesh. The entries of $\mathcal{M}_p$ are integrals of the Lagrange basis functions, but can be rewritten in terms of a quadrature rule,
\begin{equation*}
\left( \mathcal{M}_p \right)_{ij} = \langle \phi_i^{\operatorname{grad}}, \phi_j^{\operatorname{grad}} \rangle = \sum_{l=1}^{N_V} \int_{D_l} \  \phi_i^{\operatorname{grad}} \phi_j^{\operatorname{grad}} d\bm{x} \approx \sum_{l=1}^{N_V} \sum_{k=1}^{d+1} \omega_{lk} \,\,  \phi_i^{\operatorname{grad}}(x_{lk}) \phi_j^{\operatorname{grad}}(x_{lk}),
\end{equation*}
where $d$ is the dimension, the sum in $l$ is over all of the elements, $d+1$ is the number of nodes per element, and $\omega_{lk}$ and $x_{lk}$ are the quadrature weights and nodes, respectively. Typically, Gaussian quadrature is used, but we manufacture a new rule such that the choice of nodes and weights gives the approximation, $\mathcal{M}_p \approx \mathcal{D}_V$. 
 \begin{figure}[H]
  \centering
\begin{tikzpicture}[scale = 5.7, every node/.style={scale = 1}]
\draw[thick] (.5,0) to (.25, .433) to (-.25, .433) to (-.5, 0) to (-.25, -.433) to (.25, -.433) to cycle;
\draw[thick] (.25, .433) to (-.25, -.433);
\draw[thick] (-.25, .433) to (.25, -.433);
\draw[thick] (.5,0) to (-.5, 0);
\draw[dashed, red, thick] (.25, .1443) to (.25, -.1433) to (0, -.2887) to (-.25, -.1443) to (-.25, .1443) to (0, .2887) to cycle;
\draw[dashed, red, thick] (0, .2887) to (0, .433);
\draw[dashed, red, thick] (0, -.2887) to (0, -.433);
\draw[dashed, red, thick] (-.25, .1433) to (-0.375, 0.2165);
\draw[dashed, red, thick] (.25, .1433) to (0.375, 0.2165);
\draw[dashed, red, thick] (-.25, -.1433) to (-0.375, -0.2165);
\draw[dashed, red, thick] (.25, -.1433) to (0.375, -0.2165);

\fill (0,0) circle (.5pt);
\fill (.5,0) circle (.5pt);
\fill (.25, .433) circle (.5pt);
\fill (-.25, .433) circle (.5pt);
\fill (-.5, 0) circle (.5pt);
\fill (-.25, -.433) circle (.5pt);
\fill (.25, -.433) circle (.5pt);
\fill[red] (.25, .1443) circle (.5pt);
\fill[red] (.25, -.1433) circle (.5pt);
\fill[red] (0, -.2887) circle (.5pt);
\fill[red] (-.25, -.1443) circle (.5pt);
\fill[red] (-.25, .1443) circle (.5pt);
\fill[red] (0, .2887) circle (.5pt);

\node at (0, .09) {$\bm{x}_1^D$};
\node at (.57, 0) {$\bm{x}_2^D$};
\node at (-.57, 0) {$\bm{x}_5^D$};
\node at (.25, .5){$\bm{x}_7^D$};
\node at (-.25, .5){$\bm{x}_6^D$};
\node at (-.25, - .5){$\bm{x}_4^D$};
\node at (.25, -.5){$\bm{x}_3^D$};

\node at (.25, -.2) {$\bm{x}_1^V$};
\node at ( .07, -.3){$\bm{x}_2^V$};
\node at (-.25, -.2) {$\bm{x}_3^V$};
\node at (-.25, .2) {$\bm{x}_4^V$};
\node at ( .07, .3){$\bm{x}_5^V$};
\node at (.25, .2) {$\bm{x}_6^V$};

\end{tikzpicture}

      \captionsetup{format=hang}
  \caption{Two-dimensional subset of regular dual mesh (as in Figure \ref{fig:dualmeshes}). Delaunay mesh in black solid lines, Voronoi mesh in red dashed line.} \label{fig:lumping}
\end{figure}
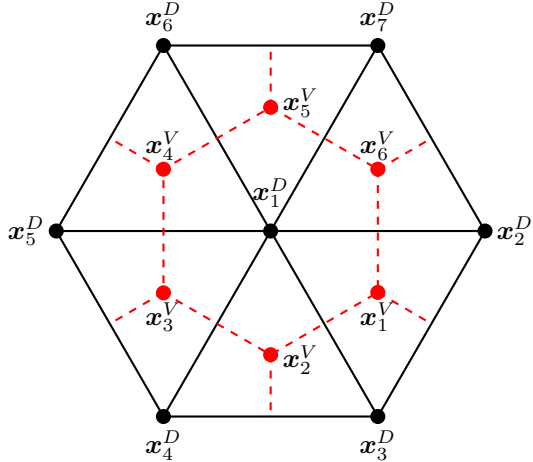

To illustrate this further, examine a subset of a regular dual mesh setup in 2D, shown in Figure \ref{fig:lumping}. On the sub-mesh, we choose the quadrature points $x_{lk}$ to be the Delaunay nodes, $\bm{x}_m^D$, and then determine what the appropriate weights should be. Consider two cases: $i = j$ and $i \neq j$.
When $i\neq j$, we have that all terms in the sum are zero, as $\phi^{\operatorname{grad}}_i(\bm{x}^D_n) = 0$ if $i \neq n$ and $\phi^{\operatorname{grad}}_j(\bm{x}^D_n) = 0$ if $j \neq n$ by the definition of the finite-element basis functions and degrees of freedom. The only way to get a nonzero term in the sum is if $i = j = n$.
When $i=j$, note that each term in the sum is only nonzero when $x_{lk} = \bm{x}_i^D$. Consider the small mesh in Figure \ref{fig:lumping}, and the case when $i=j=1$. Define the element enumeration using MFD notation, i.e., let $D_1$ be the Delaunay element defined by Voronoi node $\bm{x}_1^V$, and let $D_l$ be defined by $\bm{x}_l^V$. Let $k$ enumerate the nodes in an element in increasing order. For example, for $l=1$, we have $x_{11} = \bm{x}_1^D$, $x_{12} =\bm{x}_2^D$ and  $x_{13} =\bm{x}_3^D$ in Figure \ref{fig:lumping}.  The quadrature rule on this sub-mesh becomes,
\begin{align*}
\sum_{l=1}^{N_V} \sum_{k=1}^{d+1} \omega_{lk} \,\,  \phi_1^{\operatorname{grad}}(x_{lk}) \phi_1^{\operatorname{grad}}(x_{lk}) = & ~\omega_{11}  \phi_1^{\operatorname{grad}}(x_{11}) \phi_1^{\operatorname{grad}}(x_{11}) + \omega_{21}\phi_1^{\operatorname{grad}}(x_{21}) \phi_1^{\operatorname{grad}}(x_{21})\\
+& ~\omega_{31}\phi_1^{\operatorname{grad}}(x_{31}) \phi_1^{\operatorname{grad}}(x_{31}) + \omega_{41}\phi_1^{\operatorname{grad}}(x_{41}) \phi_1^{\operatorname{grad}}(x_{41})\\
+& ~\omega_{51}\phi_1^{\operatorname{grad}}(x_{51}) \phi_1^{\operatorname{grad}}(x_{51}) + \omega_{61}\phi_1^{\operatorname{grad}}(x_{61}) \phi_1^{\operatorname{grad}}(x_{61}) \\
=& ~\omega_{11} + \omega_{21}+\omega_{31}+\omega_{41}+\omega_{51}+\omega_{61}.
\end{align*}

To get the correct approximation by mass lumping, this sum must equal $|V_1|$, the area of the Voronoi cell defined by the six Voronoi nodes.  Therefore, we choose $\omega_{11} = |V_1 \cap D_1|$, $\omega_{21} = |V_1 \cap D_2|,$ ..., $\omega_{61} = |V_1 \cap D_6|$. Substituting in we see that,
\begin{align*}
\sum_{l=1}^{N_V} \sum_{k=1}^{d+1}\omega_{lk} \,\,  \phi_1^{\operatorname{grad}}(x_{lk}) \phi_1^{\operatorname{grad}}(x_{lk}) &= \omega_{11} + \omega_{21}+\omega_{31}+\omega_{41}+\omega_{51}+\omega_{61} \\
&= |V_1 \cap D_1| + |V_1 \cap D_2| + ... + |V_1 \cap D_6| \\
&= |V_1|.
\end{align*}
The same argument can be applied more generally to the $i=j$ case, where we want $\omega_{lk} = |D_l \cap V_i|$ when $x_{lk} = \bm{x}_i^D$. This gives us the approximation $(\mathcal{M}_p)_{ii} \approx |V_i|$. The same idea applies to the 3D case, except we have four nodes in each element and the weights represent a volume instead of an area. 

In general, for the Lagrange elements with mass matrix $\mathcal{M}_p$, we use the following quadrature rule to lump the entries onto the diagonal. Summing over elements and nodes per element,
we have for scalar functions $u, v \in H(\operatorname{grad})$,
\begin{align*}
\langle u, v \rangle_{\widetilde{\mathcal{M}_p}} := \sum_{l=1}^{N_V} \sum_{k=1}^{d+1} \omega^{\operatorname{grad}}_{lk} \,\,  u(x_{lk})v(x_{lk}), \quad
\| u \|_{\widetilde{\mathcal{M}_p}}^2 := \langle u, u \rangle_{\widetilde{\mathcal{M}_p}},
\end{align*}
where the $x_{lk}$ are given by the nodes of the Delaunay mesh, and $\omega^{\operatorname{grad}}_{lk} = |D_l \cap V_i|$ when $x_{lk} = \bm{x}_i^D$. Thus, we have that $\mathcal{M}_p \approx \widetilde{\mathcal{M}_p} = \mathcal{D}_V$, which is the standard mass-lumping scheme for $H(\operatorname{grad})$ \cite{MLBook}. 

Similarly for the RT elements, we introduce the following inner product, summing over elements and faces per element, for $\bm u, \bm v \in \bm H(\operatorname{div})$, 

\begin{align*}
\langle \bm u, \bm v \rangle_{\widetilde{\mathcal{M}}_{\bm B}} &:= \sum_{l=1}^{N_V} \sum_{k=1}^{d+1} \omega^{\operatorname{div}}_{lk} \left( \frac{1}{f_{lk}} \int_{f_{lk}} \bm u \cdot \bm{\hat{n}} \, dS\right) \left(\frac{1}{f_{lk}} \int_{f_{lk}} \bm v \cdot \bm{\hat{n}} \, dS \right),\\
\|\bm u\|_{\widetilde{\mathcal{M}}_{\bm B}}^2 &:= \langle \bm u, \bm u \rangle_{\widetilde{\mathcal{M}}_{\bm B}}.
\end{align*}
Using the same idea as the previous lumping scheme and following~\cite{RTlump}, we choose quadrature weights to be $\omega_{lk}^{\operatorname{div}} = \left| \bm{e}_{lm}^V \cap D_l \right|\,|\partial D_{lm}|$ when $f_{lk} = \partial D_{lm}$, where $\partial D_{lm}$ is the $i$th face in the Delaunay mesh enumeration. Then, the mass matrix is approximated by $\mathcal{M}_{\bm{B}} \approx \widetilde{\mathcal{M}_{\bm{B}}} = \mathcal{D}_{\partial D} \mathcal{D}_{\bm e ^V}$.

Finally, we examine the N\'ed\'elec element mass matrix, $\mathcal{M}_{\bm{E}}$, by following~\cite{Nedlump}. We have the entries computed with the following quadrature rule, for $\bm u, \bm v \in \bm H(\operatorname{curl})$,
\begin{align*}
\langle \bm u, \bm v \rangle_{\widetilde{\mathcal{M}}_{\bm E}} &:=\sum_{l=1}^{N_D} \sum_{k=1}^{\frac{d(d+1)}{2}} \omega^{\operatorname{curl}}_{lk} \left( \frac{1}{e_{lk}} \int_{e_{lk}} \bm u  \cdot \bm{\hat{t}} \, ds\right) \left(\frac{1}{e_{lk}} \int_{e_{lk}} \bm v \cdot \bm{\hat{t}} \, ds \right),\\
\|\bm u\|_{\widetilde{\mathcal{M}}_{\bm E}}^2 &:= \langle \bm u, \bm u \rangle_{\widetilde{\mathcal{M}}_{\bm E}},
\end{align*}
where edge $e_{lk} = \bm{e}_{mn}^D$ is the $i$th edge in the Delaunay mesh enumeration, and $e_{lk} = \bm{e}_{mn}^D$ is contained in Delaunay tetrahedron $D_l$. Choosing weights to be $\omega_{lk}^{\operatorname{curl}} = \left| \partial V_{mn} \cap D_l \right|\,\left| \bm{e}_{mn}^D \right|$ yields $\mathcal{M}_{\bm{E}} \approx \widetilde{\mathcal{M}_{\bm{E}} } = \mathcal{D}_{\partial V} \mathcal{D}_{\bm{e}^D}$.
\begin{remark}
The mass-lumping schemes can be modified for non-constant $\epsilon(\bm{x})$ and $\mu(\bm{x})$ by taking a constant approximation of the coefficients on each element. With the piecewise constant approximation, the quadrature weights and sparsity of the mass-lumping schemes are unchanged.
\end{remark}
Putting this all together we have the mass-lumped FE system as,
\begin{equation}
\underbrace{
	\begin{bmatrix}
\frac{2}{\tau} \widetilde{\mathcal{M}_{\bm{B}}} & \widetilde{\mathcal{M}_{\bm{B}}} \mathcal{K}^{FE} &   \\
-  \left( \mathcal{K}^{FE} \right)^T \widetilde{\mathcal{M}_{\bm{B}}} & \frac{2}{\tau} \widetilde{\mathcal{M}_{\bm{E}}} & \widetilde{\mathcal{M}_{\bm{E}}} \mathcal{G}^{FE} \\
& -  \left(\mathcal{G}^{FE} \right)^T  \widetilde{\mathcal{M}_{\bm{E}}} & \frac{2}{\tau} \widetilde{\mathcal{M}_p}  
\end{bmatrix}
}_{=:\widetilde{\mathcal{A}}_{\text{FE}}}  
\begin{bmatrix}
\bm{B}_h^n \\
\bm{E}_h^n \\
p_h^n
\end{bmatrix}
= 
\begin{bmatrix}
\bm{g}_{\bm{B}}\\
\bm{g}_{\bm{E}} \\
g_{p}
\end{bmatrix}. \label{eq:FEmasslumped}
\end{equation}

\section{Connections between MFD and FEM}
To study the well-posedness of the mimetic discretization (\ref{eq:MFDFaraday})--(\ref{MFDGauss}) and design efficient solvers for the resulting linear system, we draw connections to the FE scheme, noting that with mass lumping, the two systems have the same block structure. First, we rewrite the FE gradient and curl operators in terms of the MFD incidence matrices as follows,
\begin{align}\label{eq:FEgrad}
\mathcal{G}^{FE} = \mathcal{D}_{\bm{e}^D}^{-1} \mathcal{G}, \qquad 
\mathcal{K}^{FE} = \mathcal{D}_{\partial D}^{-1} \mathcal{K} \mathcal{D}_{\bm{e}^D},
\end{align}
where we scale incident matrices on the mesh to be consistent with the definitions in (\ref{eq:gradFEdof})--(\ref{eq:curlFEdof}).
Applying a left scaling to the mass-lumping schemes, and substituting in (\ref{eq:FEgrad}) to (\ref{eq:FEmasslumped}), gives a new scaled FE linear system,
\begin{align}\label{eq:FELinearScaled}
\underbrace{ \begin{bmatrix}
\mathcal{D}_{\partial D}^{-1} \mathcal{D}_{\bm e^V}^{-1}  & & \\
& \mathcal{D}_{\partial V}^{-1} \mathcal{D}_{\bm e^D}^{-1} & \\
& & \mathcal{D}_V^{-1}
\end{bmatrix} \widetilde{\mathcal{A}}_{\text{FE}}}_{=:\mathcal{A}_{\text{SFE}}}
\begin{bmatrix}
\bm{B}_h^n \\
\bm{E}_h^n \\
p_h^n
\end{bmatrix}& 
=  \begin{bmatrix}
\mathcal{D}_{\partial D}^{-1} \mathcal{D}_{\bm e^V}^{-1} & & \\
& \mathcal{D}_{\partial V}^{-1} \mathcal{D}_{\bm e^D}^{-1} & \\
& & \mathcal{D}_V^{-1}
\end{bmatrix}
\begin{bmatrix}
\bm{g}_{\bm{B}}\\
\bm{g}_{\bm{E}} \\
g_{p}
\end{bmatrix}.
\end{align}
Substituting in $\widetilde{\mathcal{M}_p} = D_V$, $\widetilde{\mathcal{M}_{\bm{B}}} = \mathcal{D}_{\partial D} \mathcal{D}_{\bm e ^V}$, and $\widetilde{\mathcal{M}_{\bm{E}} } = \mathcal{D}_{\partial V} \mathcal{D}_{\bm{e}^D}$, we get $\mathcal{A}_{\text{SFE}} = \mathcal{A}_{\text{MFD}}$ and recover \emph{exactly} the MFD system in (\ref{eq:MFDLinear}). 

With the above equivalence, we apply a FE well-posedness proof to the mass-lumped FE system, thus obtaining the well-posedness of the mimetic system. For simplicity in dealing with the scaled system, we also introduce the function space, $
\widetilde{\bm{V}_h} \ := \ \widetilde{\bm{H}_{h,0}}(\operatorname{div}) \times \widetilde{\bm{H}_{h,0}}(\operatorname{curl}) \times \widetilde{H_{h,0}}(\operatorname{grad})$, which is $\bm{V}_h$ with basis functions scaled to reflect the left scaling in (\ref{eq:FELinearScaled}),
\begin{center}
\begin{tabular}{ccl}
$\bm{\widetilde{H}}_{h,0}(\operatorname{div}) = \text{span}\{ \bm{\widetilde{\phi}}_i^{\operatorname{div}} \}$, & where & $\bm{\widetilde{\phi}}_i^{\operatorname{div}} = \left(\mathcal{D}_{\partial D}^{-1}\mathcal{D}_{e^V}^{-1} \right)_{ii} \bm{\phi}_i^{\operatorname{div}}$; \\
$\bm{\widetilde{H}}_{h,0}(\operatorname{curl}) = \text{span}\{ \bm{\widetilde{\phi}}_i^{\operatorname{curl}} \}$, & where & $\bm{\widetilde{\phi}}_i^{\operatorname{curl}} = \left(\mathcal{D}_{\partial V}^{-1} \mathcal{D}_{\bm{e}^D}^{-1} \right)_{ii} \bm{\phi}_i^{\operatorname{curl}}$; \\
$\widetilde{H}_{h,0}(\operatorname{grad}) = \text{span}\{ \widetilde{\phi}_i^{\operatorname{grad}} \}$, & where & $\widetilde{\phi}_i^{\operatorname{grad}} = \left(\mathcal{D}_V^{-1} \right)_{ii} \phi_i^{\operatorname{grad}}$.
\end{tabular}
\end{center}
Now, (\ref{eq:FELinearScaled}), and thus the MFD Maxwell System \eqref{eq:MFDLinear}, can be written in variational form (for simplicity, the subscripts indicating the time-step iteration and inclusion in the FE space are excluded): find $(\bm{B}, \bm{E}, p) \in \bm{V}_h$ such that for all $(\widetilde{\bm{C}},\widetilde{\bm{F}},\widetilde{q}) \in \widetilde{\bm{V}_h}$,
\begin{align}
\frac{2}{\tau} \langle \bm{B}, \widetilde{\bm{C}} \rangle_{\widetilde{\mathcal{M}}_{\bm{B}}} + \langle \nabla \times \bm{E}, \widetilde{\bm{C}} \rangle_{\widetilde{\mathcal{M}}_{\bm{B}}}  &= (\bm g_{\bm{B}}, \widetilde{\bm{C}}), \label{eq:lumpedFaraday} \\
-\langle \bm{B},  \nabla \times \widetilde{\bm{F}}  \rangle_{\widetilde{\mathcal{M}}_{\bm{B}}} + \frac{2}{\tau} \langle \bm{E}, \widetilde{\bm{F}} \rangle_{\widetilde{\mathcal{M}}_{\bm{E}}} + \langle \nabla p, \widetilde{\bm{F}} \rangle_{\widetilde{\mathcal{M}}_{\bm{E}}} &= (\bm g_{\bm{E}}, \widetilde{\bm{F}}), \label{eq:lumpedAmpere} \\
- \langle \bm{E}, \nabla \widetilde{q} \rangle _ {\widetilde{\mathcal{M}}_{\bm{E}}} + \frac{2}{\tau} \langle p, \widetilde{q} \rangle_{\widetilde{\mathcal{M}}_{p}} &= (g_{p}, \widetilde{q}). \label{eq:lumpedGauss}
\end{align}
Note that MFD in this case can be considered as a Petrov-Galerkin FEM. 

Following the notation in \cite{jxsolver}, we introduce the bilinear form,
\begin{align}
a(\bm{B}, \bm{E}, p; \widetilde{\bm{C}}, \widetilde{\bm{F}}, \widetilde{q}) \ := \ &\frac{2}{\tau} \langle \bm{B}, \widetilde{\bm{C}} \rangle_{\widetilde{\mathcal{M}}_{\bm{B}}} + \langle \nabla \times \bm{E}, \widetilde{\bm{C}} \rangle_{\widetilde{\mathcal{M}}_{\bm{B}}}  -\langle \bm{B},  \nabla \times \widetilde{\bm{F}}  \rangle_{\widetilde{\mathcal{M}}_{\bm{B}}} \label{eq:bilinearform} \\
&+ \frac{2}{\tau} \langle \bm{E}, \widetilde{\bm{F}} \rangle_{\widetilde{\mathcal{M}}_{\bm{E}}} + \langle \nabla p, \widetilde{\bm{F}} \rangle_{\widetilde{\mathcal{M}}_{\bm{B}}}  - \langle \bm{E}, \nabla \widetilde{q} \rangle _ {\widetilde{\mathcal{M}}_{\bm{E}}} + \frac{2}{\tau} \langle p, \widetilde{q} \rangle_{\widetilde{\mathcal{M}}_{p}}, \nonumber
\end{align}
and the weighted norms
\begin{align}
\| \bm{B} \|_{\operatorname{div}}^2 \ &:= \ \frac{2}{\tau} \| \bm{B}\|_{\widetilde{\mathcal{M}}_{\bm{B}}} ^2 + \| \nabla \cdot \bm{B} \| ^2, \\
\| \bm{E} \|_{\operatorname{curl}}^2 \ &:= \ \frac{2}{\tau} \| \bm{E}\|_{\widetilde{\mathcal{M}}_{\bm{E}}} ^2 + \frac{\tau}{2}\| \nabla \times \bm{E} \|_{\widetilde{\mathcal{M}}_{\bm{B}}} ^2, \\
\| p \|_{\operatorname{grad}}^2 \ &:= \ \frac{2}{\tau} \| p \|_{\widetilde{\mathcal{M}}_{p}} ^2 + \frac{\tau}{2}\| \nabla p \|_{\widetilde{\mathcal{M}}_{\bm{E}}} ^2, \\
\triplenorm{\left(\bm{B},\bm{E},p\right)}^2 \ & := \
\|\bm{B}\|_{\operatorname{div}}^2 + \|\bm{E}\|_{\operatorname{curl}}^2 + \|p\|_{\operatorname{grad}}^2.
\end{align}
The following theorem shows that (\ref{eq:lumpedFaraday})--(\ref{eq:lumpedGauss}) is well-posed, and  therefore (\ref{eq:MFDFaraday})--(\ref{MFDGauss}) is well-posed.
\begin{theorem} 
If $\bm g_{\bm{B}} \in \left( \bm{H}_{h,0} (\operatorname{div}) \right) '$, the MFD system  (\ref{eq:lumpedFaraday})--(\ref{eq:lumpedGauss}) is well-posed, namely, it satisfies the \textit{inf-sup} condition, 
\begin{equation}
\sup_{\substack{(\widetilde{\bm{C}}, \widetilde{\bm{F}}, \widetilde{q}) \in \widetilde{\bm{V_h}} \\
				(\widetilde{\bm{C}}, \widetilde{\bm{F}},  \widetilde{q}) \neq \bm 0}}	\frac{a(\bm{B}, \bm{E}, p; \widetilde{\bm{C}}, \widetilde{\bm{F}}, \widetilde{q})}{ \triplenorm{( \widetilde{\bm{C}}, \widetilde{\bm{F}},  \widetilde{q}\, )}} \geq \frac{1}{4} \triplenorm{\left( \bm{B}, 	\bm{E} , p \right)} , \label{eq:infsuporiginal}
\end{equation}
and is bounded,
{\small 
\begin{equation}
a(\bm{B}, \bm{E}, p; \widetilde{\bm{C}}, \widetilde{\bm{F}},  \widetilde{q}) \leq C \triplenorm{ \bm{B}, 	\bm{E} , p}  \,\triplenorm{ ( \widetilde{\bm{C}}, \widetilde{\bm{F}},  \widetilde{q} \,)}. \label{eq:boundednessoriginal}
\end{equation}
}
\end{theorem}
\begin{proof}
Consider a variation of (\ref{eq:bilinearform}), where the term $\langle \nabla \cdot \bm{B}, \nabla \cdot \widetilde{\bm{C}} \rangle$ is added, 
\begin{align}
 \hat{a}(\bm{B}, \bm{E}, p; \widetilde{\bm{C}},\widetilde{\bm{F}}, \widetilde{q}) \ := a(\bm{B}, \bm{E}, p; \widetilde{\bm{C}},\widetilde{\bm{F}}, \widetilde{q}) +   \langle \nabla \cdot \bm{B}, \nabla \cdot \widetilde{\bm{C}} \rangle. \label{eq:bilinearform1}
\end{align}
Since $\nabla \cdot \bm{B} = 0$ for all $t \geq 0$, (\ref{eq:bilinearform}) and (\ref{eq:bilinearform1}) are equivalent. We proceed by showing that (\ref{eq:bilinearform1}) satisfies the \textit{inf-sup} condition, 
\begin{equation}
\sup_{\bm{0} \neq (\widetilde{\bm{C}}, \widetilde{\bm{F}}, \widetilde{q}) \in \widetilde{\bm{V_h}}}	\frac{\hat{a}(\bm{B}, \bm{E}, p; \widetilde{\bm{C}}, \widetilde{\bm{F}}, \widetilde{q})}{\triplenorm{ ( \widetilde{\bm{C}}, \widetilde{\bm{F}},  \widetilde{q} \,)}} \geq \frac{1}{4} \triplenorm{ ( \bm{B}, 	\bm{E} , p )} , \label{eq:infsup}
\end{equation}
and is bounded,
{\small 
\begin{equation}
\hat{a}(\bm{B}, \bm{E}, p; \widetilde{\bm{C}}, \widetilde{\bm{F}},  \widetilde{q}) \leq C \triplenorm{ (\bm{B}, 	\bm{E} , p)} \, \triplenorm{( \widetilde{\bm{C}}, \widetilde{\bm{F}},  \widetilde{q}\,)}. \label{eq:boundedness}
\end{equation}
}
Let $\widetilde{\bm{C}} = \bm{B} + \frac{\tau}{2} \nabla \times \bm{E}$, $\widetilde{\bm{F}} = \bm{E} + \frac{\tau}{2} \nabla p$, and $ \widetilde{q} = p$. Substituting into the bilinear form (\ref{eq:bilinearform1}), using the fact that $\nabla \cdot \nabla \times \bm{E} = 0$ and $\nabla \times \nabla p = \bm 0$, and simplifying, we have
 \begin{align*}
\hat{a}(\bm{B}, \bm{E}, p; \widetilde{\bm{C}}, \widetilde{\bm{F}}, \widetilde{q}) = &\frac{2}{\tau} \|\bm{B}\|_{\widetilde{\mathcal{M}}_{\bm{B}}}^2 + \|\nabla \cdot \bm{B}\|^2 + \frac{2}{\tau}\|\bm{E}\|_{\widetilde{\mathcal{M}}_{\bm{E}}} ^2 + \frac{\tau}{2} \| \nabla \times \bm{E} \|_{\widetilde{\mathcal{M}}_{\bm{B}}}^2  \\
&+ \frac{2}{\tau} \| p \|_{\widetilde{\mathcal{M}}_{p}}^2 + \frac{\tau}{2} \| \nabla p \|_{\widetilde{\mathcal{M}}_{\bm{E}}} ^2+ \langle \nabla \times \bm{E}, \bm{B} \rangle_{\widetilde{\mathcal{M}}_{\bm{B}}} + \langle \nabla p, \bm{E} \rangle_{\widetilde{\mathcal{M}}_{\bm{E}}}.
\end{align*}
Using Young's inequality to bound the cross terms, it follows that
\begin{align*}
\hat{a}(\bm{B}, \bm{E}, p; \widetilde{\bm{C}}, \widetilde{\bm{F}}, \widetilde{q}) \ \geq \ &\frac{2}{\tau} \|\bm{B}\|_{\widetilde{\mathcal{M}}_{\bm{B}}}^2 + \|\nabla \cdot \bm{B}\|^2 + \frac{2}{\tau}\|\bm{E}\|_{\widetilde{\mathcal{M}}_{\bm{E}}} ^2 + \frac{\tau}{2} \| \nabla \times \bm{E} \|_{\widetilde{\mathcal{M}}_{\bm{B}}}^2 \\
& + \frac{2}{\tau} \| p \|_{\widetilde{\mathcal{M}}_{p}}^2 + \frac{\tau}{2} \| \nabla p \|_{\widetilde{\mathcal{M}}_{\bm{E}}} ^2  - \frac{\tau}{4} \|\nabla \times \bm{E}\|_{\widetilde{\mathcal{M}}_{\bm{B}}}^2 - \frac{1}{\tau} \|\bm{B}\|_{\widetilde{\mathcal{M}}_{\bm{B}}}^2 \\
&-\frac{1}{\tau} \|\bm{E}\|_{\widetilde{\mathcal{M}}_{\bm{E}}}^2 - \frac{\tau}{4}\| \nabla p \|_{\widetilde{\mathcal{M}}_{\bm{E}}} ^2\\
\ = \ &\frac{1}{\tau} \|\bm{B}\|_{\widetilde{\mathcal{M}}_{\bm{B}}}^2 + \|\nabla \cdot \bm{B}\|^2 + \frac{1}{\tau}\|\bm{E}\|_{\widetilde{\mathcal{M}}_{\bm{E}}} ^2 + \frac{\tau}{4} \| \nabla \times \bm{E} \|_{\widetilde{\mathcal{M}}_{\bm{B}}}^2  \\
&+ \frac{2}{\tau} \| p \|_{\widetilde{\mathcal{M}}_{p}}^2 + \frac{\tau}{4} \| \nabla p \|_{\widetilde{\mathcal{M}}_{\bm{E}}} ^2\\
\ \geq \ & \frac{1}{2} \triplenorm{( \bm{B}, 	\bm{E} , p)} ^2.
\end{align*}
Next, bound the norms of the test functions using the triangle inequality and Young's inequality.
{\small 
\begin{align*}
\triplenorm{ ( \widetilde{\bm{C}}, \widetilde{\bm{F}},  \widetilde{q} \,)}^2 \ = \ & \|\bm{B} + \frac{\tau}{2} \nabla \times \bm{E}\|_{\operatorname{div}}^2 + \|\bm{E} + \frac{\tau}{2} \nabla p\|_{\operatorname{curl}}^2 + \| p\|_{\operatorname{grad}}^2 \\
\ \leq \ & \| \bm{B}\|_{\operatorname{div}}^2 + \frac{\tau ^2}{4} \| \nabla \times \bm{E} \|_{\operatorname{div}}^2 + \tau \| \bm{B} \|_{\operatorname{div}} \| \nabla \times \bm{E} \|_{\operatorname{div}} \\
 &+ \| \bm{E} \|_{\operatorname{curl}}^2 + \frac{\tau ^2}{4} \| \nabla p \|_{\operatorname{curl}}^2 + \tau \|\bm{E} \|_{\operatorname{curl}} \| \nabla p \|_{\operatorname{curl}}  + \| p\|_{\operatorname{grad}}^2  \\
 \leq \ &
\| \bm{B}\|_{\operatorname{div}}^2 + \frac{\tau ^2}{4} \| \nabla \times \bm{E} \|_{\operatorname{div}}^2 + \| \bm{B}\|_{\operatorname{div}}^2 + \frac{\tau ^2}{4} \| \nabla \times \bm{E}\|_{\operatorname{div}}^2 \\
+&\| \bm{E} \|_{\operatorname{curl}}^2 + \frac{\tau ^2}{4} \| \nabla p \|_{\operatorname{curl}}^2 + \| \bm{E} \|_{\operatorname{curl}}^2 + \frac{\tau ^2}{4} \| \nabla p\|_{\operatorname{curl}}^2  + \| p\|_{\operatorname{grad}}^2 \\
\ \leq \ &4 \triplenorm{ (\bm{B}, 	\bm{E} , p)} .
\end{align*}
}
Weak coercivity, \eqref{eq:infsup}, follows directly. To show boundedness, \eqref{eq:boundedness}, apply Cauchy--Schwarz twice to \eqref{eq:bilinearform1}. The well-posedness of bilinear form $\hat{a}$ defined in \eqref{eq:bilinearform1} follows directly from Babuska--Brezzi theory. Since $\hat{a}$, (\ref{eq:bilinearform1}), and the original bilinear form $a$, (\ref{eq:bilinearform}), are equivalent, the scaled, mass-lumped FE system, which is equivalent to the MFD system, (\ref{eq:lumpedFaraday})--(\ref{eq:lumpedGauss}) is well-posed (similar arguments as in Lemma 1 and Theorem 8 of~\cite{huStableFiniteElement2017} give the result). This implies that (\ref{eq:bilinearform}) satisfies (\ref{eq:infsuporiginal})--(\ref{eq:boundednessoriginal}), which completes the proof. 
\end{proof}

\section{Block Preconditioners based on Exact Block Factorization} \label{sec:precond}
One of the benefits of drawing connections between MFD and FEM, in addition to the ability to show well-posedness of the MFD discretization, is that robust linear solvers developed for FEM \cite{precond2, precond1, precond3, precond4, precond5} can now be applied to the MFD system. In this work, we extend the ideas from \cite{jxsolver, precond4} to the MFD system and develop robust preconditioners based on block factorization, exploiting the structure-preserving nature of the discretization. 

 Good block preconditioners are often based on Schur complements and their approximations, and the accuracy of the approximations greatly influences the performance of the preconditioner. However, the structure-preserving discretization allows for Schur complements to be computed exactly, and the exact sequence of the discrete spaces yields sparse Schur complements that are used directly without approximation.

 To more clearly exploit the structure-preserving nature of the MFD discretization, we re-write the blocks of (\ref{eq:MFDLinear}) back in the MFD operator notation given by (\ref{eq:graddivcurlD})--(\ref{eq:graddivcurlV}),
\begin{equation}
\mathcal{A}_{\text{MFD}} ~\bm{x} = \bm{b} \iff  \begin{bmatrix}
 \frac{2}{\tau} \mathcal{I}_{\bm{e}^V} & \text{curl}_D  &    \\
 - \text{curl}_V   &  \frac{2}{\tau} \mathcal{I}_{\bm{e}^D}  & \text{grad}_D \\
  	& -\text{div}_D & \frac{2}{\tau} \mathcal{I}_V 
 \end{bmatrix}
 \begin{bmatrix}
 \bm{B}_V^n \\
 \bm{E}_D^n \\
 p_D^n
 \end{bmatrix}
 =
  \begin{bmatrix}
 \bm{g}^V_{\bm{B}} \\
\bm{g}^D_{\bm{E}} \\
 g^D_p
 \end{bmatrix}. \label{eq:simpMFD}
\end{equation}
 Recall that the structure-preserving discretization enforces the properties of the gradient, curl, and divergence (curl grad = 0 and div curl = 0) on the discrete level as $\text{curl}_D \text{grad}_D = 0$, $\text{curl}_V \text{grad}_V = 0$, $\text{div}_V \text{curl}_D = 0$, and $\text{div}_D \text{curl}_V = 0$. Exploiting these properties gives the exact block factorization of (\ref{eq:MFDLinear}), 
\begin{equation}
\mathcal{A}_{\text{MFD}} = \underbrace{\begin{bmatrix}
\mathcal{I}_{\bm{e}^V} &  & \\
-\frac{\tau}{2} \text{curl}_V & \mathcal{I}_{\bm{e}^D} & \\
& -\frac{\tau}{2}\text{div}_D & \mathcal{I}_V
\end{bmatrix}}_{\mathcal{L}}\underbrace{\begin{bmatrix}
\frac{2}{\tau}\mathcal{I}_{\bm{e}^V} & & \\
& \mathcal{S}_{\bm{E}} & \\
& & \mathcal{S}_p
\end{bmatrix}}_{\mathcal{S}}\underbrace{\begin{bmatrix}
\mathcal{I}_{\bm{e}^V} & \frac{\tau}{2} \text{curl}_D & \\
& \mathcal{I}_{\bm{e}^D} & \frac{\tau}{2} \text{grad}_D \\
& & \mathcal{I}_V
\end{bmatrix}}_{\mathcal{U}}, \label{eq:factorization}
\end{equation}
with the Schur complements computed exactly as,
\begin{align*}
\mathcal{S}_{\bm{E}} &= \frac{\tau}{2} \text{curl}_V \text{curl}_D + \frac{2}{\tau} \mathcal{I}_{\bm{e}^D}, &
\mathcal{S}_p &= \frac{\tau}{2} \text{div}_D\text{grad}_D + \frac{2}{\tau} \mathcal{I}_V.
\end{align*}
For the remainder of the paper, we drop the subscript notation and just represent $\mathcal{A}_{\text{MFD}}$ by $\mathcal{A}$.

Several block preconditioners can be designed from the exact factorization, (\ref{eq:factorization}). A natural choice of preconditioner is $\mathcal{S}^{-1}$. However, this involves computing the inverse of the Schur complements, making this choice impractical. To rectify this, replace the Schur complements, $\mathcal{S}_{\bm{E}}$ and $\mathcal{S}_p$, with good preconditioners, $\mathcal{Q}_{\bm{E}}$ and $\mathcal{Q}_p$. For example, an HX-preconditioner~\cite{HiptmairXu2007b} can be used for $\mathcal{Q}_{\bm{E}}$ and a standard multigrid preconditioner for $\mathcal{Q}_p$.  Note that the top left entry in $\mathcal{S}$ is a scaled identity matrix, so no spectrally-equivalent approximation is needed. For the remaining two blocks, we assume that
\begin{align}
 c_{1, \bm{E}} \leq  \lambda( \mathcal{Q}_{\bm{E}}\mathcal{S}_{\bm{E}})  \leq c_{2, \bm{E}}, \label{eq:specequivE} \\
c_{1,p} \leq \lambda(\mathcal{Q}_{p}\mathcal{S}_{p}) \leq c_{2, p}.  \label{eq:specequivp}
\end{align}
This implies that for $\mathcal{Q} = \text{diag}\left( \left( \frac{2}{\tau}\mathcal{I}_{\bm{e}^V} \right)^{-1},\mathcal{Q}_{\bm{E}}, \mathcal{Q}_p \right)$, we have 
\begin{equation*}
c_1 \leq \lambda(\mathcal{Q} \mathcal{S}) \leq c_2,
\end{equation*}
where $c_1 = \text{min}( c_{1,\bm{E}}, c_{1,p})$ and $c_2 = \text{max}( c_{2,\bm{E}}, c_{2,p})$. 

From this factorization, we consider three different block preconditioners,   
\begin{equation}
\mathcal{X}_{\mathcal{LS}} := \mathcal{Q} \mathcal{L}^{-1}, \quad \mathcal{X}_{\mathcal{SU}}:= \mathcal{U}^{-1} \mathcal{Q}, \quad \mathcal{X}_{\mathcal{LSU}} := \mathcal{U}^{-1} \mathcal{Q} \mathcal{L}^{-1}, \label{eq:precond}
\end{equation}
where $\mathcal{L}^{-1}$ and $\mathcal{U}^{-1}$ can be computed directly,
\begin{align*}
\mathcal{L}^{-1} &= \begin{bmatrix}
\mathcal{I}_{\bm{e}^V} & & \\
\frac{\tau}{2} \text{curl}_V & \mathcal{I}_{\bm{e}^D} & \\
& \frac{\tau}{2} \text{div}_D & \mathcal{I}_V
\end{bmatrix}, &
\mathcal{U}^{-1} &= \begin{bmatrix}
\mathcal{I}_{\bm{e}^V} & -\frac{\tau}{2} \text{curl}_D & \\
& \mathcal{I}_{\bm{e}^D} & -\frac{\tau}{2} \text{grad}_D \\
& & \mathcal{I}_V
\end{bmatrix}.
\end{align*}
 In the following theorem, we prove that these preconditioners are robust with respect to the discretization parameters. By bounding the eigenvalues for the preconditioned system, we guarantee good performance of GMRES. Note that the following proof can be done to show that the constants $c_1$ and $c_2$ are also independent of the PDE parameters, $\epsilon$ and $\mu$, but for simplicity we only consider the case where $\epsilon = \mu = 1$. Otherwise, the identity matrices in the diagonal matrix of the decomposition would be scaled by the PDE parameter values.
\begin{theorem}
Let $\mathcal{X}_{\mathcal{LS}}$, $\mathcal{X}_{\mathcal{SU}}$, and $\mathcal{X}_{\mathcal{LSU}}$ be defined by (\ref{eq:precond}) and assume the spectral equivalent properties (\ref{eq:specequivE})--(\ref{eq:specequivp}) hold. Then, 
\begin{equation}
\lambda\left(\mathcal{X}_{\mathcal{LS}}\mathcal{A}\right) \in [c_1, c_2], \quad \lambda\left(\mathcal{X}_{\mathcal{SU}} \mathcal{A}\right) \in [c_1, c_2], \quad \lambda\left(\mathcal{X}_{\mathcal{LSU}} \mathcal{A} \right) \in [c_1, c_2],
\end{equation}
where $c_1 = \textup{min} \left(c_{1,\bm{E}}, c_{1,p} \right)$ and $c_2 = \textup{max}\left( c_{2,\bm{E}}, c_{2,p}\right)$ are constants that do not depend on discretization parameters, $h$ and $\tau$.

\end{theorem}
\begin{proof}
First, consider $\mathcal{X}_{\mathcal{LS}} \mathcal{A}$,
\begin{equation*}
\mathcal{X}_{\mathcal{LS}} \mathcal{A} = \mathcal{Q} \mathcal{L}^{-1} \mathcal{LSU} = \mathcal{Q} \mathcal{SU} = \begin{bmatrix}
\mathcal{I}_{\bm{e}^V} & \frac{\tau}{2} \text{curl}_D & \\
& \mathcal{Q}_{\bm{E}} \mathcal{S}_{\bm{E}} & \mathcal{Q}_{\bm{E}}\text{grad}_D \\
& & \mathcal{Q}_p \mathcal{S}_p
\end{bmatrix}.
\end{equation*}
The eigenvalues $\lambda\left(\mathcal{X}_{\mathcal{LS}} {\mathcal{A}} \right)$ are determined by the eigenvalues of the diagonal blocks since $\mathcal{X}_{\mathcal{LS}}\mathcal{A}$ is block upper triangular. The first block is an identity matrix, whose eigenvalues are all ones. For the other two blocks, we use \eqref{eq:specequivE}--\eqref{eq:specequivp} to bound the eigenvalues and, overall, have~$\lambda\left(\mathcal{X}_{\mathcal{LS}}\mathcal{A}\right) \in [c_1, c_2]$.

To bound the eigenvalues of $\mathcal{X}_{\mathcal{SU}} \mathcal{A}$, consider the eigenvalue problem,
\begin{equation*}
\mathcal{X}_{\mathcal{SU}} \mathcal{A} \bm{x} = \lambda \bm{x} \iff \mathcal{A} \bm{x} = \lambda \mathcal{X}_{\mathcal{SU}}^{-1} \bm{x}.
\end{equation*}
By substituting in the decompositions of $\mathcal{A}$ and $\mathcal{X}_{\mathcal{SU}}$, it follows that,
\begin{equation*}
\mathcal{LSU} \bm{x} = \lambda \mathcal{Q}^{-1} \mathcal{U} \bm{x}.
\end{equation*}
Let $\bm{y} = \mathcal{U} \bm{x}$ and left multiply by $\mathcal{Q}$. Then,
$\mathcal{Q} \mathcal{L} {\mathcal{S}} \bm{y} = \lambda \bm{y}$,
where
\begin{equation*}
\mathcal{Q} \mathcal{L}{\mathcal{S}}  = \begin{bmatrix}
\mathcal{I}_{\bm{e}^V} &  & \\
-\mathcal{Q}_{\bm{E}} \text{curl}_V & \mathcal{Q}_{\bm{E}} \mathcal{S}_{\bm{E}} & \\
&  {- \mathcal{Q}_p \text{div}_D} & \mathcal{Q}_p \mathcal{S}_p
\end{bmatrix}.
\end{equation*}
Since $\mathcal{Q} \mathcal{L}{\mathcal{S}} $ is block lower triangular, again the eigenvalues only depend on the diagonal blocks. Thus, by (\ref{eq:specequivE})--(\ref{eq:specequivp}), $\lambda\left(\mathcal{X}_{\mathcal{SU}}\mathcal{A}\right) \in [c_1, c_2].$

Finally, consider the preconditioned system
\begin{equation*}
\mathcal{X}_{\mathcal{LSU}} \mathcal{A} \bm{x} = \mathcal{U}^{-1} Q \mathcal{L}^{-1} \mathcal{L} \mathcal{S} \mathcal{U} \bm{x}  = \mathcal{U}^{-1} Q  \mathcal{S} \mathcal{U} \bm{x} = \lambda \bm{x}.
\end{equation*}
Left multiplying by $\mathcal{U}$ and letting $\bm{y} = \mathcal{U} \bm{x}$ yields,
\begin{equation*}
Q \mathcal{S} \bm{y} = \lambda \bm{y},
\end{equation*}
where
\begin{equation*}
Q\mathcal{S} = \begin{bmatrix}
\mathcal{I}_{\bm{e}^V} & & \\
& \mathcal{Q}_{\bm{E}} \mathcal{S}_{\bm{E}} & \\
& & \mathcal{Q}_p \mathcal{S}_p
\end{bmatrix}.
\end{equation*}
By the same reasoning as the other two cases, we conclude $\lambda\left(\mathcal{X}_{\mathcal{LSU}}\mathcal{A}\right) \in [c_1, c_2]$.
\end{proof}

\subsection{Preservation of the Divergence-Free Magnetic Field}
The goal of using a structure-preserving discretization for the Maxwell system is to enforce the PDE constraints on the discrete level. However, even if the discretization holds, the approximate solve for each time step of a linear solver can destroy these properties. The following theorem ensures that at each iteration of the linear solver, the divergence-free condition for $\bm{B}$ is preserved. 
\begin{theorem}
Let $\bm{x}^0 = \left( \bm{B}_V^0, \bm{E}_D^0, p_D^0 \right)^T$ be the initial guess for the MFD system satisfying $\operatorname{div}_V \bm{B}_V^0 = 0$, and let $\bm{b} = \left( \bm{g}_{\bm{B}}^V, \bm{g}_{\bm{E}}^D, g_p^D \right)^T$ be the MFD right-hand side satisfying $\operatorname{div}_V  \bm{g}_{\bm{B}}^V = 0$. Then all iterations, $\bm{x}^k = \left( \bm{B}_V^k, \bm{E}_D^k, p_D^k \right)^T$, of the preconditioned GMRES method satisfy $\operatorname{div}_V \bm{B}_V^k = 0$, where $\mathcal{X}$ is any of the preconditioners defined in (\ref{eq:precond}).  
\end{theorem}
\begin{proof}
Define the Krylov subspace,
\begin{equation*}
K^k(\mathcal{X} \mathcal{A}, \bm{r}^0) = \text{span} \{ \bm{r}^0, \mathcal{X}\mathcal{A} \bm{r}^0, \left( \mathcal{X} \mathcal{A} \right)^2 \bm{r}^0, ..., \left( \mathcal{X} \mathcal{A} \right)^k \bm{r}^0 \} 
\end{equation*}
with $\bm{r}^0 = \left( \bm{r}_{\bm{B}}^0, \bm{r}_{\bm{E}}^0, \bm{r}_p^0 \right)^T := \mathcal{X} \left( \bm{b} - \mathcal{A} \bm{x}^0 \right)$ such that for each iteration,
\begin{equation}
\bm{x}^k \in \bm{x}^0 + K^k \left( \mathcal{X} \mathcal{A}, \bm{r}^0 \right). \label{eq:xk}
\end{equation}
By the assumption of the divergence-free initial data, it follows that $\text{div}_V\bm{r}_{\bm{B}}^0 = 0$.
Next, let $\bm{v}^m = \left( \bm{v}_{\bm{B}}^m, \bm{v}_{\bm{E}}^m, \bm{v}_p^m \right)^T := \left( \mathcal{X}\mathcal{A} \right)^m \bm{r}^0$ for $m = 0,1,...,k-1$. Since $\bm{v}^m = \mathcal{X} \mathcal{A} \bm{v}^{m-1}$, we compute $\bm{v}_{\bm{B}}^m$ for each of the preconditioners in (\ref{eq:precond}). Setting $\mathcal{X} = \mathcal{X}_{\mathcal{LS}}$,
\begin{equation*}
\bm{v}_{\bm{B}}^m = \left( \frac{2}{\tau} \mathcal{I}_{\bm{e}^V} \right)^{-1} \left( \frac{2}{\tau}  \bm{v}_{\bm{B}}^{m-1} + \text{curl}_D \bm{v}_{\bm{E}}^{m-1} \right) = \bm{v}_{\bm{B}}^{m-1} + \frac{\tau}{2} \text{curl}_D \bm{v}_{\bm{E}}^{m-1}.
\end{equation*} 
For $\mathcal{X} = \mathcal{X}_{\mathcal{SU}}$,
\begin{align*}
\bm{v}_{\bm{B}}^m 
&= \bm{v}_{\bm{B}}^{m-1} + \frac{\tau}{2} \text{curl}_D \left( \bm{v}_{\bm{E}}^{m-1} + \mathcal{Q}_{\bm{E}} \text{curl}_V \bm{v}_{\bm{B}}^{m-1} - \frac{2}{\tau} \mathcal{Q}_{\bm{E}} \bm{v}_{\bm{E}}^{m-1} - \mathcal{Q}_{\bm{E}}\text{grad}_D \bm{v}_p^{m-1} \right),
\end{align*}
and for $\mathcal{X} = \mathcal{X}_{\mathcal{LSU}}$,
\begin{align*}
\bm{v}_{\bm{B}}^{m} 
=& ~\bm{v}_{\bm{B}}^{m-1}  \\
\ & \ + \text{curl}_D \big( \frac{\tau}{2} \bm{v}_{\bm{E}}^{m-1} - \frac{\tau^2}{4} \mathcal{Q}_{\bm{E}} \text{curl}_V \text{curl}_D \bm{v}_{\bm{E}}^{m-1} - \mathcal{Q}_{\bm{E}} \bm{v}_{\bm{E}}^{m-1} - \frac{\tau}{2} \mathcal{Q}_{\bm{E}} \text{grad}_D \bm{v}_p^{m-1} \big).
\end{align*}

Applying $\text{div}_V$ to $\bm{v}_{\bm{B}}^m$ for each preconditioner above, all terms with $\text{curl}_D$ in front are zero since $\text{div}_V \text{curl}_D = 0$. Then, $\text{div}_V \bm{v}_{\bm{B}}^{m} = 0$ if $\text{div}_V \bm{v}_{\bm{B}}^{m-1} = 0$. By an inductive argument, since $\text{div}_V \bm{r}_{\bm{B}}^0 = 0 $, we have that $\text{div}_V \bm{v}_{\bm{B}}^m = 0$. 
By (\ref{eq:xk}) and the definition of $\bm{v}^m$, $\bm{x}^k$ is a linear combination of $\bm{v}^m$, $m = 0, 1, ..., k-1$ . This implies that $\bm{B}_V^k$ is a linear combination of $\bm{v}_{\bm{B}}^m$. Since $\text{div}_V \bm{v}_{\bm{B}}^m = 0$, then $\text{div}_V \bm{B}_V^k = 0$ for all $k$. 
\end{proof}

\section{Numerical Results}\label{sec:num}
To demonstrate the theoretical results presented in the previous sections, consider the following test problem with essential Dirichlet boundary conditions,
{\small
\begin{align}
\bm E (\bm x, t) &= \frac{1}{\pi} e^{-t} \begin{bmatrix}
- \cos(\pi x_1) \sin( \pi x_2) \sin(\pi x_3) \\
\sin( \pi x_1) \cos( \pi x_2) \sin(\pi x_3) \\
0
\end{bmatrix}, \label{eq:testE} \\ 
\bm B(\bm x, t)&= e^{-t} \begin{bmatrix}
-\sin(\pi x_1) \cos(\pi x_2) \cos(\pi x_3) \\
-\cos(\pi x_1) \sin(\pi x_2) \cos(\pi x_3) \\
2 \cos(\pi x_1) \cos(\pi x_2) \sin(\pi x_3)
\end{bmatrix}, \label{eq:testB} \\ 
p(\bm x, t) &= 0, \label{eq:testp}\\ 
\bm j(\bm x, t) &= -e^{-t}\left( \frac{1}{\pi} + 3 \pi \right) \begin{bmatrix}
\cos (\pi x_1) \sin(\pi x_2) \sin(\pi x_3) \\
- \sin( \pi x_1) \cos( \pi x_2) \sin( \pi x_3) \\ 0
\end{bmatrix}. \label{eq:testj}
\end{align}}

While the analysis presented in this paper holds for a general Voronoi mesh, for simplicity, we consider a non-degenerate mesh, where the Voronoi points do not lie on the boundary or outside of the corresponding Delaunay tetrahedra. To get a non-degenerate Voronoi mesh, we design a Delaunay triangulation for the FE domain that consists of a cube with a rectangular pyramid on each face (see Figure~\ref{fig:mesh}).  
The pyramids are defined by the faces of the unit cube and the points $\left(- \frac{1}{2},  \frac{1}{2},  \frac{1}{2} \right)$,  $\left( \frac{3}{2},  \frac{1}{2},  \frac{1}{2} \right)$, $\left(\frac{1}{2},  -\frac{1}{2},  \frac{1}{2} \right)$, $\left(\frac{1}{2},  \frac{3}{2},  \frac{1}{2} \right)$, $\left(\frac{1}{2},  \frac{1}{2},  -\frac{1}{2} \right)$, $\left(\frac{1}{2},  \frac{1}{2},  \frac{3}{2} \right)$.  Uniform refinement is used to get more resolved meshes and Table \ref{table:mesh} lists the geometric information for the different resolutions considered. Numerical experiments are implemented in the 
HAZmath package~\cite{hazmath} written by the authors. All timed numerical results are done using a
workstation with an 8-core 3-GHz Intel Xeon Sandy Bridge CPU and 256 GB of RAM.

\begin{figure}[H]
	\centering
	\includegraphics[scale=.23]{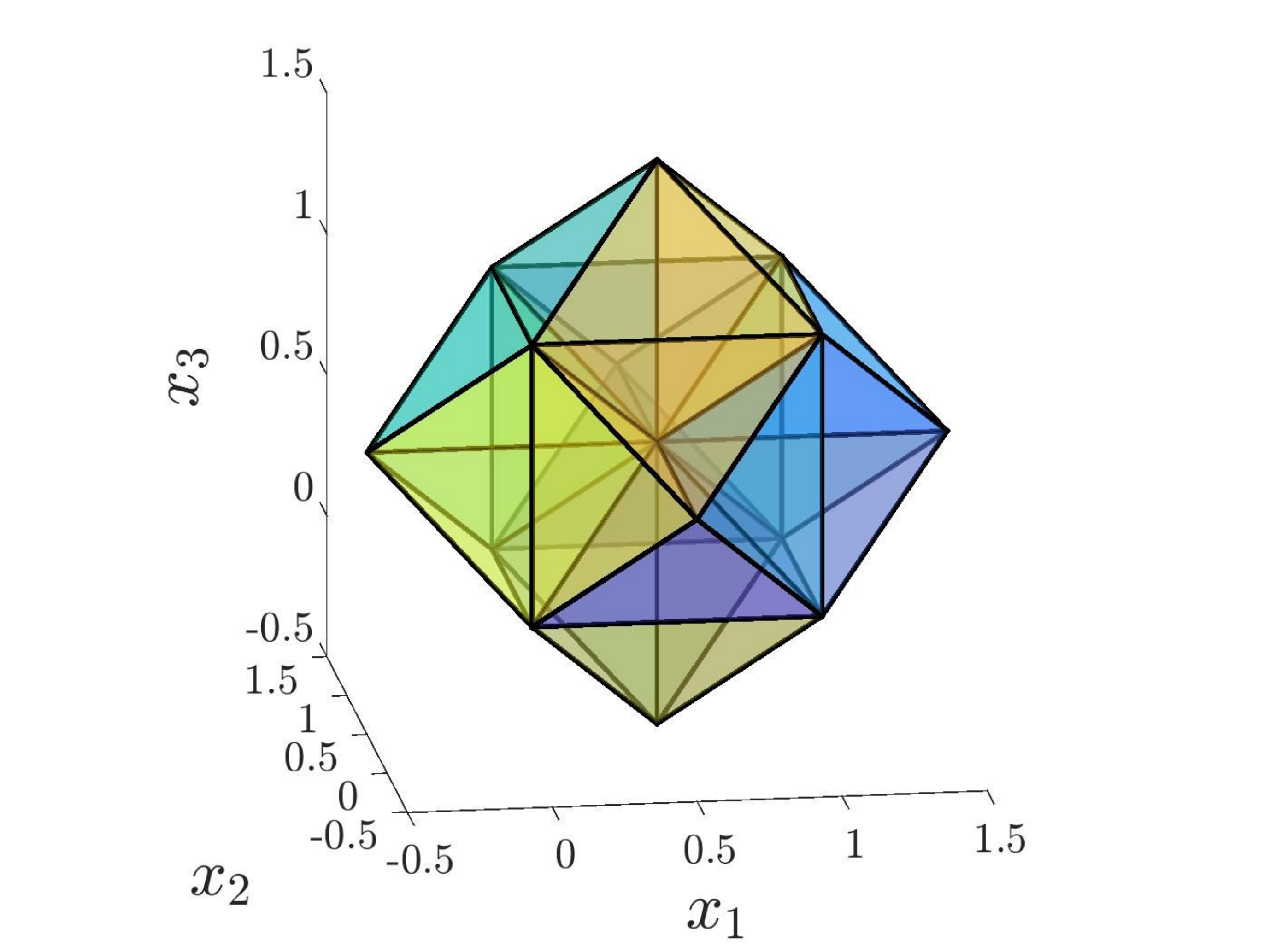}  \qquad
	\includegraphics[scale=.2]{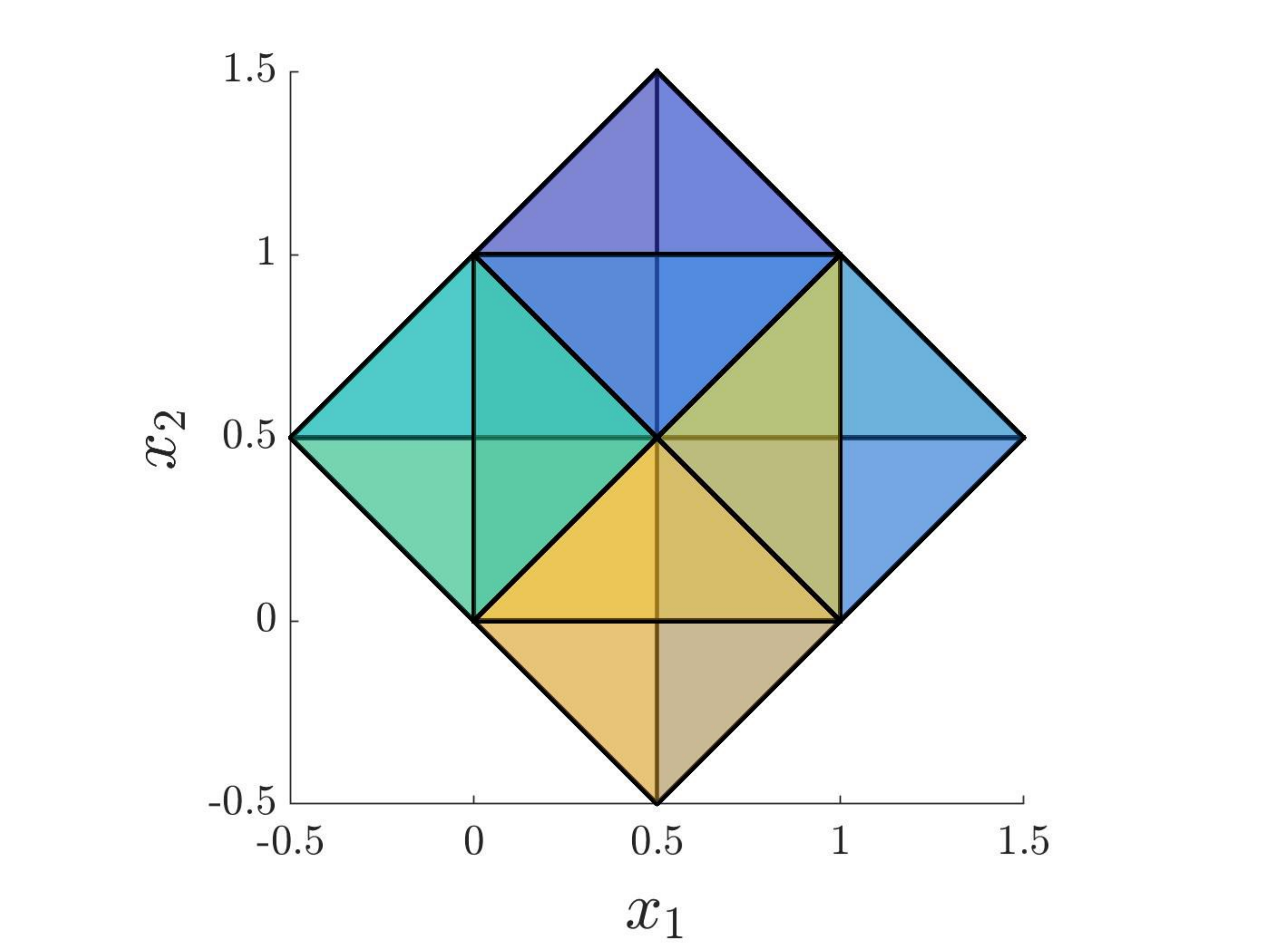} 
	\caption{Left: Delaunay mesh with $h = 1$; Right: cross section at $z=1/2$} \label{fig:mesh}
\end{figure}

\begin{table}[H]
\centering
\begin{tabular}{|c||ccccc|}
\hline 
 & $h$ & Vertices & Edges & Faces & DoF \\ 
\hline 
\hline 
Mesh 1 & $1/4$ & 369 & 2,096 & 3,264 & 5,729 \\ 
\hline 
Mesh 2 & $1/8$ & 2,465 & 15,520 & 25,344 & 43,329 \\ 
\hline 
Mesh 3 & $1/16$ & 17,985 & 119,360 & 199,680 & 337,025 \\ 
\hline 
Mesh 4 & $1/32$ & 137,345  & 936,064 & 1,585,152 & 2,658,561 \\ 
\hline 
Mesh 5 & $1/64$ & 1,073,409  & 7,414,016 & 12,632,064 & 21,119,489 \\ 
\hline 
\end{tabular} 
\caption{Geometric information for the Delaunay meshes.} \label{table:mesh}
\end{table}

To demonstrate the equivalence of the two discretizations, MFD and FEM are implemented and solved with $\mathcal{X}_{\mathcal{LSU}}$-preconditioned FGMRES with restart after 100 iterations to a relative residual tolerance of $10^{-8}$. We expect FE convergence rates with respect to mesh size for both. More precisely, with $L^2$ norm in space and $L^{\infty}$ norm in time, $\bm{E}$ and $\bm{B}$ modeled with lowest order N\'ed\'elec and Raviart-Thomas elements, respectively, are expected to converge with $\mathcal{O} (h + \tau^2)$, where $h$ is the mesh partitioning and $\tau$ is the time-step size. Results for $p$ are excluded as it is just an auxiliary variable with no physical relevance to the problem. By choosing a small time step and few time iterations ($\tau = .0125$ for $8$ time steps, i.e., final time $t=0.1$), the spatial error dominates.  This is confirmed in Figure~\ref{fig:error}, where the spatial convergence of the MFD scheme is identical to the FEM one.

\begin{figure}[h!]
	\centering
\includegraphics[scale=.22]{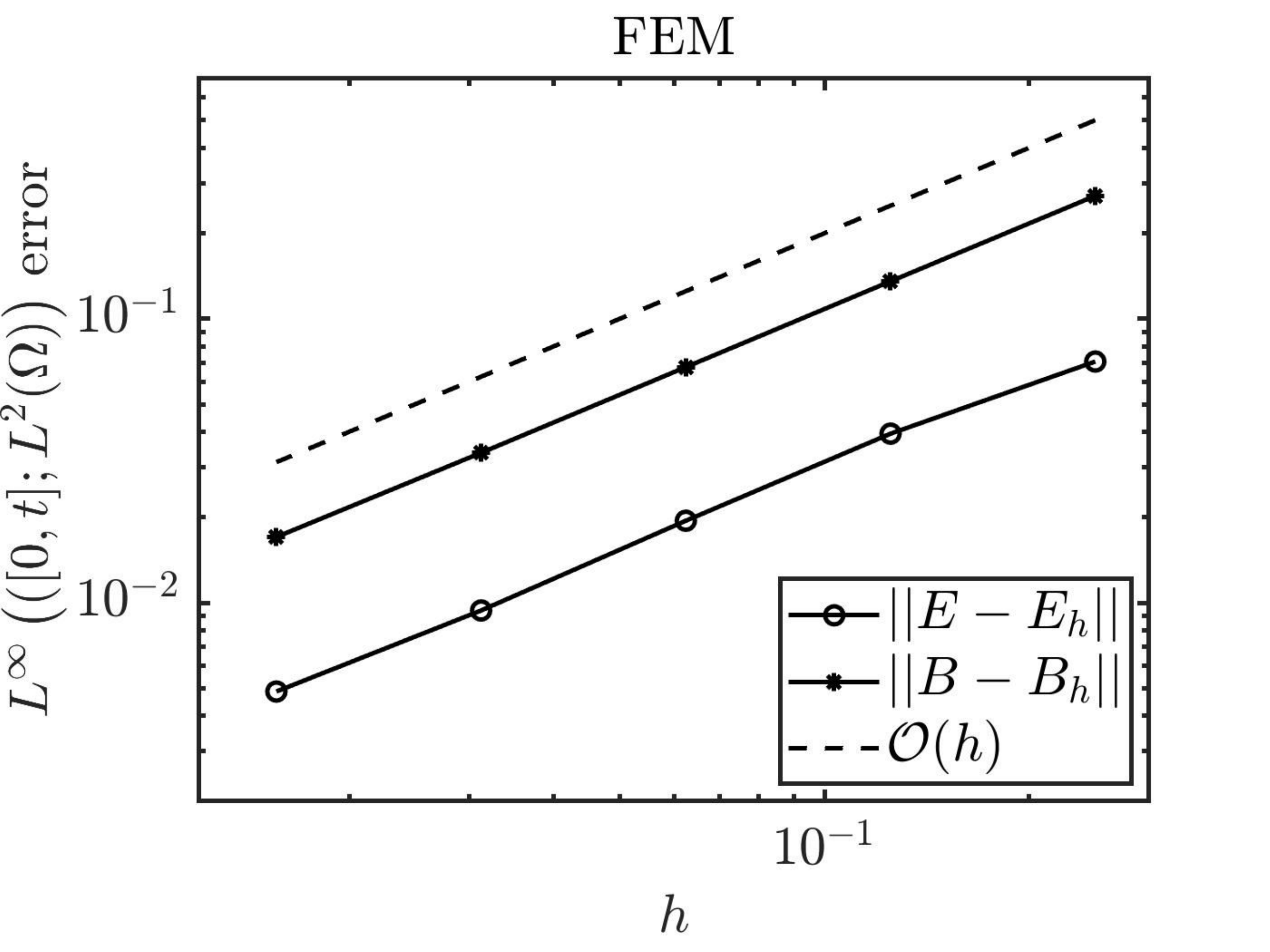}  \quad 
\includegraphics[scale=.22]{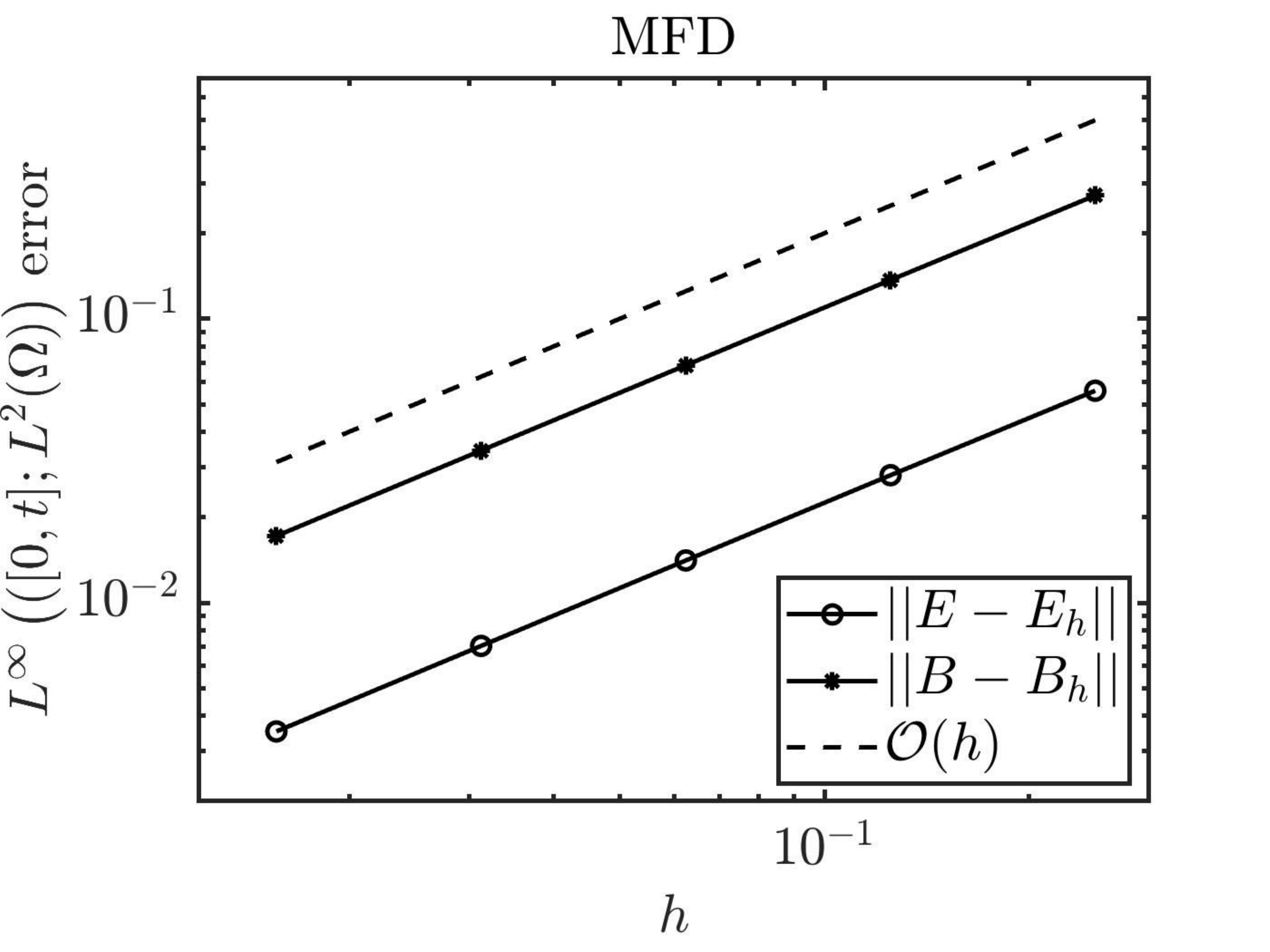} 
\caption{$L^{\infty} \left ([0,t]; L^2(\Omega) \right)$ error for FEM (left) and MFD (right) for $\tau = 0.0125$ after 8 time steps ($t = 0.1$) solving with $\mathcal{X}_{\mathcal{LSU}}$ preconditioner and FGMRES.}
\label{fig:error}
\end{figure}

Next, performance comparisons are made between the MFD and FE methods using the preconditioners developed in Section \ref{sec:precond} for MFD and  in \cite{jxsolver} for FEM. Again using test problem (\ref{eq:testE})--(\ref{eq:testj}) and meshes given by Table \ref{table:mesh}, we test the preconditioners defined in (\ref{eq:precond}) for robustness with respect to time step, $\tau$, and mesh size, $h$.  The diagonal blocks of the preconditioners are solved inexactly by preconditioned GMRES with a relative residual reduction set at $10^{-2}$ for the outer iteration. Flexible GMRES is used as the outer iteration with a relative residual stopping criteria of $10^{-8}$, with restart after 100 iterations. Iteration counts for both MFD and FEM with preconditioners $\mathcal{X}_{\mathcal{LS}}$, $\mathcal{X}_{\mathcal{SU}}$, and $\mathcal{X}_{\mathcal{LSU}}$ are reported in Tables \ref{table:MFD} and \ref{table:FEM}. In all experiments, iteration counts are averaged over the number of timesteps needed to reach $t=1$. The spectral equivalent approximations used in $\mathcal{Q}$ are an HX preconditioner~\cite{HiptmairXu2007b,KolevVassilevski2009} for $\mathcal{S}_{\bm{E}}$, and a standard algebraic multigrid method for $\mathcal{S}_p$.

The results show little variation in iteration count with varying parameters, indicating that the block preconditioners based on exact block factorization are effective and robust with respect to mesh size, $h$, and time step size, $\tau$, for both MFD and FEM as expected.  Furthermore, MFD has comparable iteration counts to FE.

\begin{table}[H]
\begin{center}
\begin{tabular}{|>{\raggedright\arraybackslash}p{1cm}|fffff||fffff||fffff|}

\hline
& \multicolumn{5}{c@{\hskip\secpad}||}{$\mathcal{X}_{\mathcal{LS}}$}
& \multicolumn{5}{c@{\hskip\secpad}||}{$\mathcal{X}_{\mathcal{SU}}$}
& \multicolumn{5}{c@{\hskip\secpad}|}{$\mathcal{X}_{\mathcal{LSU}}$} \\ \hline
\diagbox[innerwidth=1cm,innerleftsep=.25cm,innerrightsep=.25cm]{$\tau$}{$h$}
& $\frac14$ & $\frac18$ & $\frac1{16}$ & $\frac1{32}$ & $\frac1{64}$
& $\frac14$ & $\frac18$ & $\frac1{16}$ & $\frac1{32}$ & $\frac1{64}$
& $\frac14$ & $\frac18$ & $\frac1{16}$ & $\frac1{32}$ & $\frac1{64}$ \\ \hline
\multicolumn{16}{c}{} \\[-1.25em] \hline
0.2     & 5 & 5 & 5 & 6 & 8 & 5 & 5 & 5 & 6 & 6 & 4 & 3 & 4 & 4 & 4 \\ \hline
0.1     & 3 & 4 & 5 & 5 & 6 & 4 & 5 & 5 & 5 & 6 & 2 & 3 & 3 & 4 & 4 \\ \hline
0.05    & 4 & 3 & 4 & 4 & 5 & 4 & 4 & 5 & 5 & 5 & 3 & 2 & 3 & 3 & 3 \\ \hline
0.025   & 3 & 4 & 3 & 4 & 4 & 3 & 4 & 4 & 5 & 5 & 2 & 3 & 2 & 3 & 3 \\ \hline
0.0125  & 2 & 3 & 4 & 3 & 3 & 3 & 3 & 4 & 4 & 4 & 2 & 2 & 3 & 2 & 3 \\ \hline
\end{tabular}
\caption{Iteration counts for the block preconditioners based on block factorization for MFD. Left: block lower triangular, $\mathcal{X}_{\mathcal{LS}}$; Center: block upper triangular, $\mathcal{X}_{\mathcal{SU}}$; Right: symmetric, $\mathcal{X}_{\mathcal{LSU}}$. } \label{table:MFD}
\end{center}
\end{table}

\begin{table}[H]
\begin{center}
\begin{tabular}{|>{\raggedright\arraybackslash}p{1cm}|fffff||fffff||fffff|}
\hline
& \multicolumn{5}{c@{\hskip\secpad}||}{$\mathcal{X}_{\mathcal{LS}}$}
& \multicolumn{5}{c@{\hskip\secpad}||}{$\mathcal{X}_{\mathcal{SU}}$}
& \multicolumn{5}{c@{\hskip\secpad}|}{$\mathcal{X}_{\mathcal{LSU}}$} \\ \hline
\diagbox[innerwidth=1cm,innerleftsep=.25cm,innerrightsep=.25cm]{$\tau$}{$h$}
& $\frac14$ & $\frac18$ & $\frac1{16}$ & $\frac1{32}$ & $\frac1{64}$
& $\frac14$ & $\frac18$ & $\frac1{16}$ & $\frac1{32}$ & $\frac1{64}$
& $\frac14$ & $\frac18$ & $\frac1{16}$ & $\frac1{32}$ & $\frac1{64}$ \\ \hline
\multicolumn{16}{c}{} \\[-1.25em] \hline
0.2    & 4 & 4 & 4 & 4 & 4 & 5 & 5 & 5 & 5 & 4 & 3 & 3 & 3 & 3 & 3 \\ \hline
0.1    & 5 & 4 & 4 & 4 & 4 & 5 & 5 & 5 & 5 & 4 & 3 & 3 & 3 & 3 & 3 \\ \hline
0.05   & 4 & 4 & 4 & 4 & 3 & 5 & 4 & 4 & 4 & 4 & 3 & 2 & 3 & 3 & 3 \\ \hline
0.025  & 3 & 3 & 3 & 3 & 3 & 4 & 4 & 3 & 3 & 2 & 3 & 2 & 2 & 2 & 2 \\ \hline
0.0125 & 3 & 3 & 3 & 3 & 3 & 3 & 3 & 4 & 3 & 3 & 3 & 2 & 2 & 2 & 3 \\ \hline
\end{tabular}
\caption{Iteration counts for the block preconditioners based on block factorization for FEM. Left: block lower triangular, $\mathcal{X}_{\mathcal{LS}}$; Center: block upper triangular, $\mathcal{X}_{\mathcal{SU}}$; Right: symmetric, $\mathcal{X}_{\mathcal{LSU}}$.} \label{table:FEM}
\end{center}
\end{table}

To further compare the two methods, CPU solve time per time iteration is examined for fixed time-step size, as well as the time scaling of the solve time per time step for the three preconditioned FGMRES solvers, both with respect to mesh refinement. In Table \ref{table:time1}, the average solve time over ten iterations with fixed $\tau$ is reported. We see that as the number of degrees of freedom increases, the FE method beats MFD, especially for $\mathcal{X}_{\mathcal{LS}}$ and $\mathcal{X}_{\mathcal{SU}}$. This is further demonstrated in Figure \ref{fig:barplot}, where a side-by-side comparison of the three preconditioners for MFD and FEM is given. This time disparity is likely given by the fact that MFD takes slightly more GMRES iterations per solve, which could be due to the fact that the MFD system is not symmetric while FEM is.

\begin{table}[H]
\begin{center}
\begin{tabular}{|m{0.32in}|>{\centering\arraybackslash}p{0.33in}|>{\centering\arraybackslash}p{0.65in}>{\centering\arraybackslash}p{0.65in}>{\centering\arraybackslash}p{0.65in}>{\centering\arraybackslash}p{0.65in}>{\centering\arraybackslash}p{0.65in}|}
\hline
& $h$
& $\frac14$ & $\frac18$ & $\frac1{16}$ & $\frac1{32}$ & $\frac1{64}$ \\[0.2em] \hline
\multirow{3}{*}{MFD} & $\mathcal{X}_{\mathcal{LS}}$  & $1.32\times 10^{-2}$  & $1.06 \times 10^{-1}$ & $9.95 \times 10^{-1}$ & $1.31\times10^1$ & $1.68\times10^2$ \\ 
					& $\mathcal{X}_{\mathcal{SU}}$  & $1.48 \times 10^{-2}$ & $1.12 \times 10^{-1}$ & $1.10 \times 10^0$ & $1.41 \times10^1$ & $1.96 \times10^2$ \\
					& $\mathcal{X}_{\mathcal{LSU}}$ & $1.06 \times 10^{-2}$ & $8.32 \times 10^{-2}$ & $8.15 \times 10^{-1}$ & $1.12 \times 10^1$ & $1.41 \times 10^2$ \\
\hline \hline
\multirow{3}{*}{FEM} & $\mathcal{X}_{\mathcal{LS}}$  & $2.05 \times 10^{-2}$ & $1.08 \times 10^{-1}$ & $8.84 \times 10^{-1}$ & $9.62 \times 10^0$ & $1.09\times 10^2$ \\ 
					& $\mathcal{X}_{\mathcal{SU}}$  & $2.45 \times 10^{-2}$  & $1.24 \times 10^{-1}$ & $1.02 \times 10^0$ & $1.19 \times 10^1$ & $1.17 \times 10^2$ \\
					& $\mathcal{X}_{\mathcal{LSU}}$ & $1.79 \times 10^{-2}$ & $9.17 \times 10^{-2}$ & $7.57\times 10^{-1}$ & $8.04 \times 10^0$ & $1.06 \times 10^2$ \\
\hline
\end{tabular}
\caption{Average CPU solve time per time iteration over ten time steps with $\tau = 0.1$ for MFD and FEM with all preconditioners. } \label{table:time1}
\end{center}
\end{table}

\begin{figure}[H]
	\centering
\includegraphics[scale=.3]{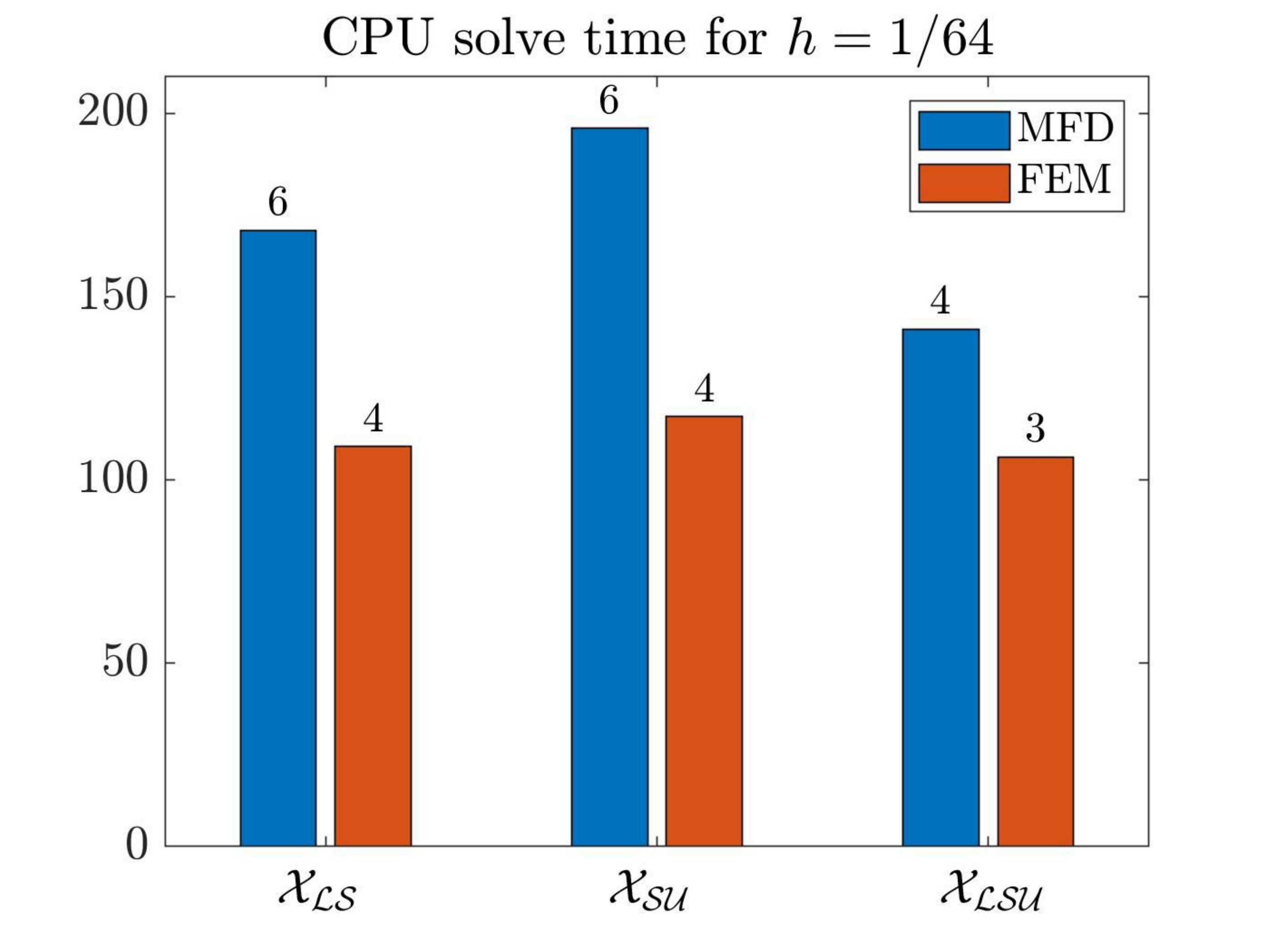} 
\caption{Comparison of solve time averaged over ten time steps of $\tau = 0.1$ for MFD and FEM for fixed mesh size $h = \frac{1}{64}$. The number of FGMRES iterations is given above each bar.}
\label{fig:barplot}
\end{figure}

Finally, it is expected that the solve time scales $\mathcal{O}\left( N \log(N) \right)$, where $N$ is the number of degrees of freedom. Figure \ref{fig:timescale} verifies this, where both FE and MFD results follow the trend of the reference line for all three preconditioners. 

Based on the results presented, we conclude that all three preconditioners are robust and effective for solving the MFD system, and give comparable results to the FE method. In terms of both solve time and iteration counts, $\mathcal{X}_{\mathcal{LSU}}$ is the best preconditioner for both methods, which was also concluded in \cite{jxsolver} for the FE method.

\begin{figure}
	\centering
\includegraphics[scale=.25]{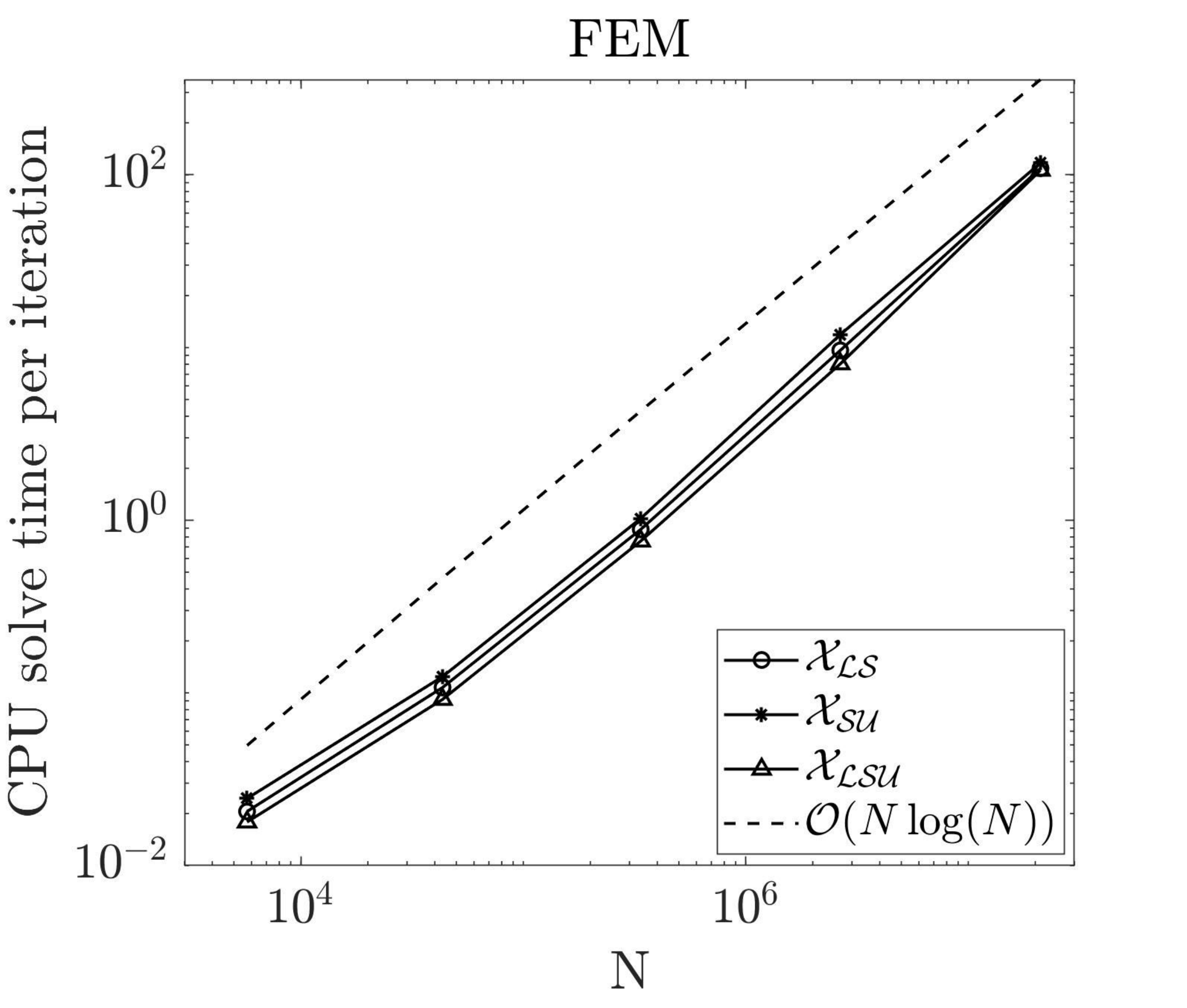}  \quad 
\includegraphics[scale=.25]{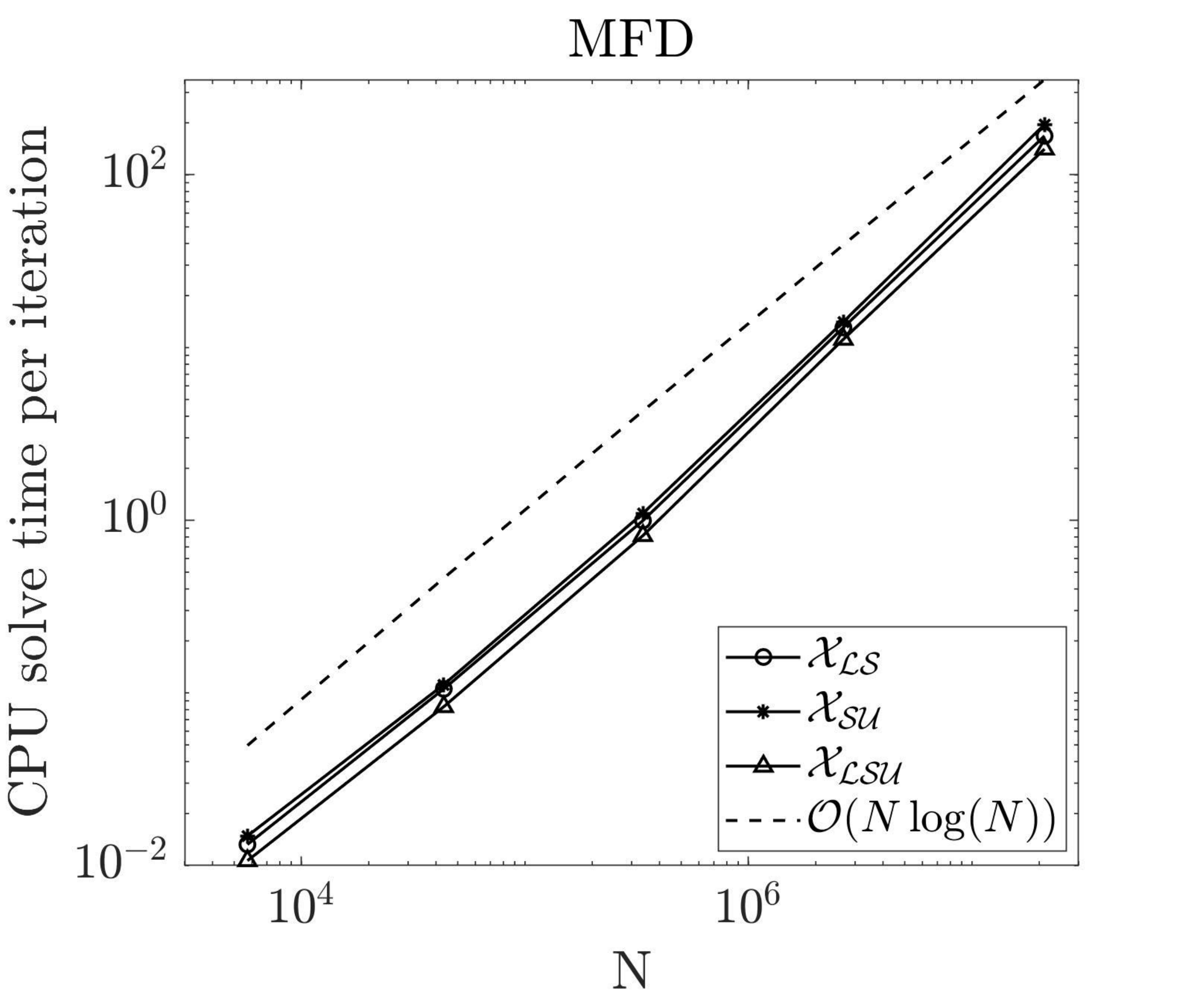} 
\caption{Time scaling of solve time averaged over ten time steps of $\tau = 0.1$ for FEM(left) and MFD(right), where $N$ is the number of degrees of freedom. }
\label{fig:timescale}
\end{figure}

\section{Conclusions}\label{sec:conc}
By examining MFD in a FE framework, we are able to exploit the FEM tools and theory to strengthen the MFD method for Maxwell. From this equivalence of the two methods, well-posedness of the MFD system is proven with Babuska-Brezzi theory. Numerical results demonstrate that the FE convergence theory is recoverable in the MFD implementation. Furthermore, in showing well-posedness of the Maxwell MFD discretization and the connection to the FEM, robust block preconditioners developed for the FEM in~\cite{jxsolver} are adopted for solving the MFD linear system efficiently. When using structure-preserving discretizations, the goal is to enforce the PDE constraints, particularly a divergence-free magnetic field for Maxwell's equations, at all time steps and all solve iterations. The block preconditioners for GMRES developed here are shown to be robust and guarantee all properties of the discretization at each time step.

All of the results in this paper apply to lowest--order FE and MFD methods. While higher--order MFD methods exist \cite{Castillo_high_MFD, Lipkinov_high_MFD, MFDart3}, connections to higher--order FEM are unclear and require further investigation. If such relationships are found, the analysis presented here should be valid. Additionally, a priori error estimates for MFD for Maxwell could be derived in a fashion similar to the techniques used for Darcy flow in \cite{MLBrezzi}. 
Furthermore, now that a FE framework for MFD has been developed for the full Maxwell system, other electromagnetic applications can be explored such as MHD, time-harmonic Maxwell, or $\bm{H}(\operatorname{curl})$ and $\bm{H}(\operatorname{div})$ problems, in general. The structure-preserving nature of the MFD discretization opens the door for many more physical applications with PDE constraints to be explored.

\section*{Acknowledgments}
The work of JHA and XH was partially funded by National Science Foundation grant DMS-1620063. The work of LTZ was supported in part by NSF DMS-1720114 and DMS-1819157.

\bibliographystyle{siam}
\bibliography{references}
\end{document}